\documentclass[11pt,notitlepage,twoside,a4paper]{amsart}
 \usepackage{amsfonts}

\usepackage{amsmath,amssymb,enumerate}

\usepackage{epsfig,fancyhdr,color}

\usepackage{amssymb}
\usepackage{amsmath,amsthm}
\usepackage{latexsym}
\usepackage{amscd}
\usepackage{psfrag}
\usepackage{graphicx}
\usepackage[latin1]{inputenc}
\usepackage[all]{xy}
\usepackage[mathcal]{eucal}

\definecolor{NoteColor}{rgb}{1,0,0}

\renewcommand{\textsc}{\textcolor{red}}

\newtheorem{theorem}{\rm\bf Theorem}[section]
\newtheorem{proposition}[theorem]{\rm\bf Proposition}
\newtheorem{lemma}[theorem]{\rm\bf Lemma}
\newtheorem{corollary}[theorem]{\rm\bf Corollary}
\newtheorem*{theorem 1}{\rm\bf Proposition 1}
\newtheorem*{theorem 2}{\rm\bf Proposition 2}

\theoremstyle{definition}
\newtheorem{definition}[theorem]{\rm\bf Definition}

\theoremstyle{remark}
\newtheorem{remark}[theorem]{\rm\bf Remark}

\newtheorem{example}[theorem]{\rm\bf Example}

\def\interieur#1{\mathord{\mathop{\kern 0pt #1}\limits^\circ}}

\title[Timelike Funk and Hilbert geometries]{Timelike Hilbert and Funk geometries}

\author{Athanase Papadopoulos}

\address{Athanase Papadopoulos,  Universit{\'e} de Strasbourg and CNRS,
7 rue Ren\'e Descartes,
 67084 Strasbourg Cedex, France
 and 
  Tata Institute of Fundamental Research, Dr. Homi Bhabha Road, Navy Nagar, Near Navy Canteen, Mandir Marg, Colaba, Mumbai, Maharashtra 400005, India.}

\email{athanase.papadopoulos@math.unistra.fr}

\author{Sumio Yamada}

 \address{Sumio Yamada,
  Gakushuin University, 1-5-1 Mejiro, Toshima, Tokyo, 171-8588, Japan}
 \email{yamada@math.gakushuin.ac.jp}

\date{\today}

\begin{document}

  \begin{abstract}
  
A timelike space is a Hausdorff topological space equipped with a partial order relation $<$ and a distance function $\rho$ satisfying a collection of axioms including a set of compatibility conditions between the partial order relation and the distance function. The distance function is defined only on a subset of the product of the space with itself that contains the diagonal, namely, $\rho(x,y)$ is defined if and only if $x<y$ or $x=y$.  Distances between pairs of distinct points in a triple $x,y,z$, whenever these distances are defined, satisfy the so-called \emph{time inequality}, which is a reverse triangle inequality $\rho(x,y)+\rho(y,z)\leq \rho(z,y)$. 

In the 1960s, Herbert Busemann developed an axiomatic  theory of timelike spaces and of locally timelike spaces. His motivation comes from the geometry underlying the theory of relativity, and he tried to adapt to this setting his geometric theory of metric spaces, namely, his theory of $G$-spaces (geodesic spaces). The classical example he considers is the $n$-dimensional Lorentzian space. Two other interesting  classes of examples of timelike spaces he introduced are the timelike analogues of the Funk and Hilbert geometries. In this paper, we investigate these two geometries, and in doing this, we introduce variants of them, in particular the timelike relative Funk and Hilbert geometries, in the Euclidean and spherical settings. We describe the Finsler infinitesimal structure of each of these geometries (with an appropriate notion of Finsler structure) and we display the interactions among the Euclidean and timelike spherical geometries. In particular, we characterize the de Sitter geometry as a special case of a timelike spherical Hilbert geometry.

{\bf The final version of this paper will appear in \emph{Differential Geometry and its Applications}}
   \medskip

  
  \medskip
  
  \end{abstract}
     
  \medskip
     
   \maketitle
  
  \tableofcontents

\section{Introduction}\label{intro}
 A timelike space is a Hausdorff topological space $\Omega$ equipped with a partial order relation $<$ and a distance function $\rho$ which plays the role of a metric. 
 The distance $\rho(x,y)$  is defined only for pairs $(x,y)\in \Omega\times \Omega$ satisfying $x\leq y$ (that is, either $x=y$ ot $x<y$) and it satisfies the following three axioms:
 \begin{enumerate}
 \item $\rho (x,x)=0$ for every $x$ in $\Omega$;
 \item $\rho(x,y) > 0$ for every $x$ and $y$ in $\Omega$ such that $x<y$;
 \item  $\rho(x,y)+\rho(y,z)\leq \rho(x,z)$ for all triples of points $x,y,z$ in $\Omega$ satisfying $x<y<z$.
 \end{enumerate}
 The last property is a  reversed triangle inequality. It is called the \emph{time inequality}.
 The distance function is asymmetric in the sense that $\rho(x,y)$ is not necessarily equal to $\rho(y,x)$. (In general, if $\rho(y,x)$ is defined, $\rho(x,y)$ is not defined unless $x=y$.)
 
 The notion of timelike space was introduced by Herbert Busemann in his memoir \emph{Timelike spaces} \cite{B-Timelike}. The distance function $\rho$ and the partial order relation $<$ satisfy an additional set of axioms including  
compatibility conditions with respect to each other. For instance, it is required that every neighborhood of a point $q$ in $\Omega$ contains points $x$ and $y$ satisfying $x<q<y$. This axiom and others, which are complex and numerous, are stated precisely in the memoir \cite{B-Timelike} by Busemann. Even though they are not made explicit in the present paper, in all the cases considered here, they will be satisfied. As a matter of fact, in the present paper, the topological space $\Omega$ will always be a subset of the Euclidean space $\mathbb{R}^n$, the sphere $S^n$ or the hyperbolic space $\mathbb{H}^n$.

The theories of timelike spaces,  timelike $G$-spaces, 
locally timelike spaces and locally timelike $G$-spaces were initiated by Busemann as analogues of his geometric theories of metric spaces and of $G$-spaces that he developed in his book \cite{G} and in other papers and monographs.
The motivation for the study of timelike spaces comes from the geometry underlying the physical theory of relativity. The classical example is the $(3+1)$-dimensional Minkowski space-time, which Busemann generalized, in his paper \cite{B-Timelike}, to the case of general timelike distance functions on finite-dimensional vector spaces which become, under a terminology that we use, \emph{timelike Minkowski spaces}. As other interesting  examples of timelike spaces, Busemann introduced  timelike analogues of the Funk and Hilbert geometries. In the present paper, we investigate several closely interrelated geometries, to which we give the names of timelike Euclidean Funk geometry, timelike Euclidean relative Funk geometry, timelike Euclidean Hilbert geometry,  
timelike hyperbolic Funk geometry,  timelike spherical relative   Funk geometry, and timelike spherical Hilbert geometry.  We establish results concerning their geodesics, their convexity properties and  their infinitesimal structure. We show in particular that they are \emph{ timelike Finsler spaces}. This means that the distance between two points is defined at the infinitesimal level by a timelike norm, that is, there exists a timelike Minkowski structure on the tangent space at each point of our space $\Omega$ such that the distance between two points is the supremum of the set of lengths of piecewise $C^1$ paths joining them, where the length of a piecewise $C^1$ path is defined using the timelike norms on vector spaces. We also give a description of the usual de Sitter space as a special case of a timelike spherical Hilbert geometry.

Busemann's interest, as well as the authors' in the subject, stem from Hilbert's Fourth Problem \cite{Pap-H} where Hilbert proposed a systematic study of metrics defined on subsets of the Euclidean space whose geodesics coincide with the Euclidean line segments. The best known, and most important example of such a metric space is the Beltrami-Klein model of the hyperbolic plane. The hyperbolic  geometry in that context is very much hinged with convex Euclidean geometry.  The aim of the current investigation is to revisit the aspect of convex geometry in the \textit{exterior region} of  convex sets  in the constant curvature spaces. This naturally produces timelike geometries as exemplified by the de Sitter geometry.  

In what follows, we will  set up a collection of necessary tools to capture the geometry of the exterior region of convex sets, and consequently reformulate the basis of timelike geometry in a way  inspired from Busemann's approach in \cite{B-Timelike}, although we differ from him at several points. 

To end this introduction, let us note that geometries like timelike geometry, in which there are naturally defined cones representing the future, where ordered triples of points satisfy  the reverse triangle inequality, and where a particular example is the geometry of the exterior of the hyperbolic disc, are topics that date back to the turn of the twentieth century; cf. Poincar\'e's paper \cite{Poin} which contains a germ of the idea, and the more explicit paper by Eduard Study \cite{Study}. One may also mention the work of A. D. Alexandrov and his school on chronogeometry, see e.g. \cite{Alex1, Alex2}, where axiomatic approaches to the geometry of Minkowski space have been investigated. Busemann's doctoral student J. Beem also worked on the subject \cite{BeemPhd, Beem}.  More recently, the subject of timelike geometry (without the name) has become the object of extensive research among low-dimensional geometers and topologists, see e.g. the works  \cite{DGK, Gold} and the surveys \cite{FS, Sch}. In all these references though, the geometry is associated with conics (or quadrics) in projective space, and not to more general convex sets as in Busemann's timelike geometry. We also refer to the recent work of Minguzzi \cite{Minguzzi}. The passage from a conic to a general convex set is comparable to the passage from hyperbolic geometry to the more general Hilbert geometry. For an exposition of Busemann's theory of timelike spaces, the reader may also refer to \cite{PY1}.

\section{The timelike Euclidean  Funk geometry}\label{Funk}

We first introduce some preliminary notions and establish some basic facts. With few exceptions, we shall use Busemann's notation in \cite{B-Timelike}, and we first recall it. 

Let $K$ be a convex hypersurface in ${\mathbb R}^n$, that is, the boundary of an open (possibly unbounded) convex set $I\subset \mathbb{R}^n$. If $K$ is not a hyperplane, it bounds a unique open convex set $I$, namely, the unique convex connected component of $\mathbb{R}^n\setminus K$. If $K$ is a hyperplane, the two connected components of $\mathbb{R}^n\setminus K$ are both convex (they are open half-spaces of $\mathbb{R}^n$), and in this case we make a choice of one of them. We call the open convex set $I$ associated with $K$ the \emph{interior} of $K$.  We denote the closure $K \cup I$ of $I$ by $\overline{K}$. 

Let  ${\mathcal P}$ be the set of supporting hyperplanes of $K$, that is, the hyperplanes $\pi$ having nonempty intersection with $K$ and such that the open convex set $I$ is contained in one of the two connected components of $\mathbb{R}^n\setminus \pi$. 

We let ${\mathbb P}$ be the set of hyperplanes in $\mathbb{R}^n$ that do not intersect the open convex set $I$.  We have ${\mathbb P} \supset {\mathcal P}$.    

To every element $\pi  \in {\mathbb P}$, we let $H^+_\pi$ be the open half-space bounded by the hyperplane $\pi$ and containing $I$, and $H^-_\pi$ the open half-space bounded by $\pi$ and not containing $I$. 
We have:
\[
I = \cap_{\pi \in {\mathcal P}} H^+_\pi = \cap_{\pi \in {\mathbb P}} H^+_\pi. 
\]
We set  \[\Omega= \mathbb{R}^n\setminus \overline{K}.\]
Then, we also have
\[
\Omega = \cup_{\pi \in {\mathcal P}} H^-_\pi.
\]

For $p$ in $\Omega$, we denote by ${\mathbb P}(p)$ the set of  hyperplanes $\pi \in {\mathbb P}$ that separate the open convex set $I$ from $p$.   
In other words, we have 
\begin{equation}
{\mathbb P}(p) = \{\pi \in {\mathbb P} \,\, | \,\, p \in H^-_\pi \}.
\end{equation} 
We also introduce the set of supporting hyperplanes separating $p$ from $I$,
\begin{equation}
 {\mathcal P}(p) = {\mathbb P}(p)  \cap {\mathcal P}. 
\end{equation}

We define $\tilde  {\mathcal P}(p)$ to be the set of supporting hyperplanes containing $p$.

\begin{definition}[Order relation]
We introduce a partial order relation between points of $\Omega$.  For any two distinct points $p$ and $q$ in $\Omega$, we write
\[
p < q 
\]
if the following three properties are satisfied:
\begin{enumerate}
\item The Euclidean ray $R(p, q)$ from $p$ through $q$ intersects the hypersurface $K$;
 \item $R(p, q)$ does not belong to a supporting hyperplane of $K$;
 \item the closed Euclidean segment $[p,q]$ does not interesect $K$.
 \end{enumerate}
\end{definition}
When $p<q$, we say that $q$ \emph{lies in the future of} $p$ and that $p$ \emph{lies in the past of} $q$ (see Figure \ref{future}).   We write $p\leq q$ if either $p<q$ or $p=q$.

\begin{figure}[!ht] 
\centering
 \psfrag{p}{\small $p$}
  \psfrag{f}{\small $\mathfrak{I}^+ (p)$: the future of $p$} 
 \psfrag{s}{\small $\mathfrak{I}^- (p)$: the past of $p$} 
\includegraphics[width=0.55\linewidth]{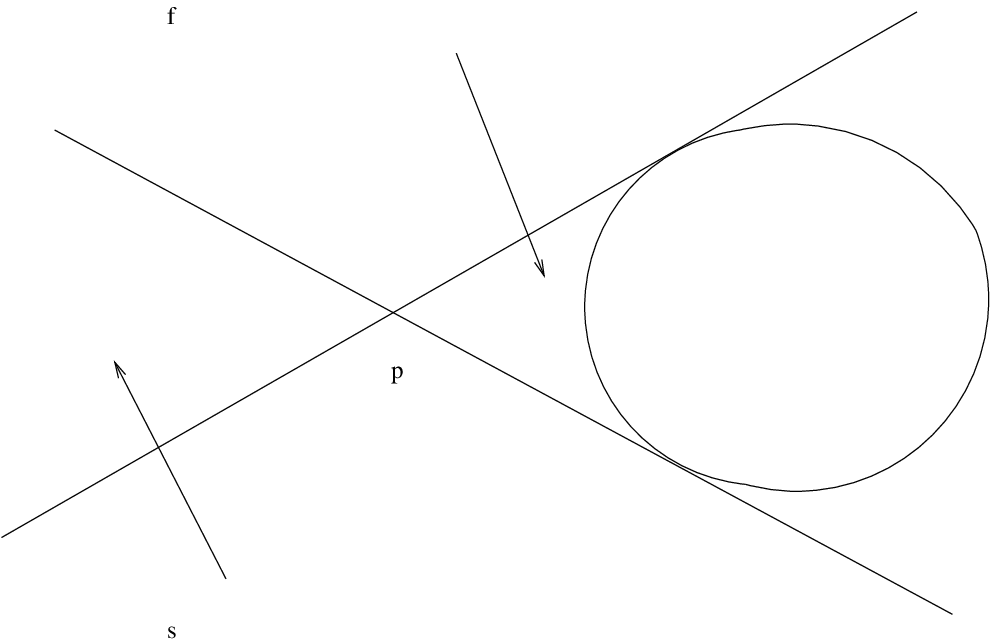}    \caption{\small {The future and the past of a point $p$}}   \label{future}  
\end{figure}


We denote by $\Omega_<$ (resp. $\Omega_{\leq}$) the set of ordered pairs $(p,q)$ in $\Omega\times\Omega$ 
satisfying $p<q$  (resp. $p\leq q$).   The set  $\Omega_<$ is disjoint from the diagonal set $\{(x, x) \,\,| \,\, x \in 
\Omega\} \subset \Omega\times\Omega$.
   
We define the following sets  that encode properties of timelike geometry.

 \begin{definition}[The future set of a point]
 For $p$ in $\Omega$, its  \emph{future set}, which we denote by  $\mathfrak{I}^+(p)\subset\Omega$,  is the set of points $q\in \Omega$ that satisfy $p<q$. 
 \end{definition}
   
For every point $p$ in $\Omega$, its future set $\mathfrak{I}^+(p)$ is nonempty, open and connected.   
 \begin{definition}[The future set of $p$ in $K$]
 For $p$ in $\Omega$, its  \emph{future set in $K$}, which we denote by $K(p)$, is the set of $k\in K$ such that the open Euclidean segment $]p,k[$ is not contained in any supporting hyperplane of $K$, and $]p,k[\cap K=\emptyset$. 
 \end{definition}

Hence  the future set of $p$ is  the union of the  open line segments $]p, k[$ where $k \in K(p)$:
\[
\mathfrak{I}^+(p)= \cup_{k \in K(p)} ]p, k[.
\]
Let 
\[{\mathbb P}(p)^c = {\mathbb P} \setminus {\mathbb P}(p)\]
where ${\mathbb P}(p)$ is as before the set of hyperplanes that separate the open convex set $I$ from $p$. Hence, for any element $\pi$ in ${\mathbb P}(p)^c$, $p$ either belongs to the hyperplane $\pi$ or to the open half-space $H^+_\pi$ containing $I$. 

The following three observations are easy to prove:

\begin{itemize}
 \item Every closed Euclidean segment $[p, k]$ for $k \in K(p)$ lies in $\overline{H_\pi^+}$ for all $\pi$ in  ${\mathbb P}(p)^c$.  

\item The  intersection  $\cap_{\pi \in {\mathbb P}(p)^c} \overline{H_\pi^+} $ is the closure of the convex hull of $K\cup p$.

\item  We have the following description of the
 the future set $\mathfrak{I}^+(p)$ of $p$:   
\[
\mathfrak{I}^+(p) = {\rm Int}\left(\cap_{\pi \in {\mathbb P}(p)^c} \overline{H_\pi^+}\right)  \setminus \overline{K},
\]
where $\mathrm{Int}(\quad)$ denotes the interior of a set. 
\end{itemize}

\begin{definition}[The future cone of a point]
We call the set  
\[
 \cup_{k \in K(p)} \{ p + t v  \,\, | \,\,  t >0, p + \lambda v = k  \mbox{ for some } \lambda > 0 \mbox{ and } k \in K(p)  \} 
\]
 the \emph{future cone} of $p$, and denote it by $\mathfrak{C}^+(p)$. 
  \end{definition}
  Ths future cone of a point is an open subset of $\mathbb{R}^n$.
  The future set $\mathfrak{I}^+(p)$ of $p$ is naturally a subset of the future cone $\mathfrak{C}^+(p)$ of $p$.

  \begin{definition}[The past set of a point]
 For $p\in \Omega$, the \emph{past set of} $p$, denoted by ${\mathfrak I}^-(p)$,  is the set of points $q$ in $\Omega$ such that $p$ is in the future of $q$.
  \end{definition}
 
 Equivalently, ${\mathfrak I}^-(p)$ is the set
\[
{\rm Int} \Big( \cup_{k \in K(p)} \{ p - t v  \,\, | \,\,  t  >0, p + \lambda v = k \mbox{ for some } \lambda > 0 \mbox{ and } k \in K(p) \} \Big).
\]
It is also an open subset of $\mathbb{R}^n$. It has a natural cone structure, and hence we also call this set the \emph{past cone} of $p$.

The set ${\mathfrak I}^-(p)$   is  also  characterized by the following: 
\begin{equation}\label{char:II}
{\mathfrak I}^-(p) = \mathrm{Int}\left(\cap_{\pi \in {\mathbb P}(p)} \overline{H^-_\pi} \right),
\end{equation}
since each ray $\{p - tv\}$ of the set ${\mathfrak I}^-(p)$ lies entirely in  $\overline{H^-_\pi}$ for all $\pi$ in ${\mathbb P}(p)$.

 We have $$ \overline{{\mathfrak I}^-(p)}  = \cap_{\pi \in {\mathbb P}(p)} \overline{H^-_\pi}.$$
 
 The sets  $ \overline{{\mathfrak I}^-(p)}$ and $\overline{{\mathfrak C}^+(p)}$ are both closed convex cones with common apex $p$.  The boundary cones of these sets are unions of rays, each contained in a supporting hyperplane of $K$ containing $p$, that is, an element of $\tilde {\mathcal P}(p)$.

The following proposition makes a relation between the partial order relation $<$ and the past cones: 
\begin{proposition} \label{2.1}
For any two points $p$ and $q$ in $\Omega$, we have 
\[
p< q \Leftrightarrow \overline{{\mathfrak I}^-(p)} \subset {\mathfrak I}^-(q).
\]
\end{proposition}

\begin{proof} 
As the first step towards proving that $p< q$ $\Rightarrow$  $\overline{{\mathfrak I}^-(p)} \subset {\mathfrak I}^-(q)$, we show the following.

\begin{lemma}\label{lem:7}
For any two points $p$ and $q$ in $\Omega$, we have 
\[
p< q \Rightarrow {\mathbb P}(p) \supset {\mathbb P}(q).
\]
\end{lemma}
\begin{proof}[Proof of the lemma] 
Suppose $p < q$.  We claim that every $\pi \in {\mathbb P}(q)$ is an element of ${\mathbb P}(p)$.  Indeed, $R(p,q)$ intersects $K$ at a point $k\in K(p)$, $k\in H_{\pi}^+$, $q\in H_{\pi}^-$. Therefore, $\pi$ intersects $R(p,q)$ at a unique point in $]q,k[$, so $p\in H_{\pi}^-$ and $\pi\in \mathbb{P}(p)$. 
This proves Lemma \ref{lem:7}.

 \end{proof}
 
 We continue the proof of Proposition \ref{2.1}.
 
 The inclusion ${\mathbb P}(p) \supset {\mathbb P}(q)$ implies 
 \[
  \cap_{\pi \in {\mathbb P}(p)} \overline{H^-_\pi} \subset  \cap_{\pi \in {\mathbb P}(q)} \overline{H^-_\pi},
 \]
 hence $\overline{{\mathfrak I}^-(p)} \subset  \overline{{\mathfrak I}^-(q)}$. 
 
 In order to complete the proof of $p< q \Rightarrow \overline{{\mathfrak I}^-(p)} \subset {\mathfrak I}^-(q)$, it remains to show that any point in $\partial \overline{{\mathfrak I}^-(p)}$ lies in ${\mathfrak I}^-(q)$.  We resort to the following general geometric fact.

\begin{lemma}\label{nesting-cone}
Suppose that  two closed convex cones $C_1$ and $C_2$ in $\mathbb{R}^n$ with respective apices $p_1$ and $p_2$  are nested: $C_1 \subset C_2$. We also assume $p_1 \neq p_2$.  Then, if a point $x \in \partial C_1$ distinct from $p_1$ lies in $\partial C_2$, then $p_1$ also lies in $\partial C_2$.   
\end{lemma}

\begin{proof}[Proof of the lemma]
As $p_1 \in C_2$, the only thing to be checked is that $p_1$ does not lie in the interior of the closed cone $C_2$.  However this cannot be the case, since by definition $p\in\mathcal{I}^-(q)$.
\end{proof} 

By applying the lemma to the situation where $C_1 = \overline{{\mathfrak I}^-(p)}$ and $C_2 = \overline{{\mathfrak I}^-(q)}$, and $x$ being a point in $\partial \overline{{\mathfrak I}^-(p)}$, we conclude that if 
$\partial \overline{{\mathfrak I}^-(p)}\not\subset{\mathfrak I}^-(q)$, then
$p$ lies in the boundary cone $\partial \overline{{\mathfrak I}^-(q)}$.  However, this cannot be the case, since then $[p,q] \subset\partial \overline{{\mathfrak I}^-(q)}$, which in turn implies that $R(p, q)$ lies in a supporting hyperplane of  $\tilde {\mathcal P}(q)$, contradicting the hypothesis $p< q$.

The implication $p< q \Leftarrow \overline{{\mathfrak I}^-(p)} \subset {\mathfrak I}^-(q)$ is immediate. As the apex $p$ of the closed convex cone $\overline{{\mathfrak I}^-(p)}$ lies in ${\mathfrak I}^-(q)$, $p$ lies in the past of $q$.

\end{proof}

\begin{corollary}\label{p:inclusion}
For any two points $p$ and $q$ in $\Omega$, we have 
\[
p< q \Rightarrow {\mathcal P}(p) \supset {\mathcal P}(q).
\]
\end{corollary}

\begin{proof} 
Let $p<q$. By Lemma \ref{lem:7}, ${\mathbb P}(p) \supset {\mathbb P}(q)$. It follows that $ {\mathcal P}(p) = ({\mathbb P}(p) \cap {\mathcal P}) \supset ({\mathbb P}(q) \cap {\mathcal P}) =  {\mathcal P}(q)$.  
\end{proof}

Note that the strict inclusion in Corollary \ref{p:inclusion} cannot be expected, as observed from the following example in ${\mathbb R}^2$:
\[
K = \{(x, y) \,\, | \,\, y = |x| \}, \,\,\, p=(0, -2) < q=(0, -1)
\]
where we have ${\mathcal P}(p) = { \ \mathcal P}(q) = \{ \mathrm{planes of equation} y = mx \mbox{ with } |m| \leq 1 \}$.

\begin{corollary}\label{cor:three}
 Let $p,q,r$ be three points in $\Omega$. If $p < q$ and $q < r$, then $p < r$.
\end{corollary}

\begin{proof}  Proposition \ref{2.1} gives:
\[
p< q \mbox{ and } q < r \Leftrightarrow  \overline{{\mathfrak I}^-(p)} \subset {\mathfrak I}^-(q)  \subset \overline{{\mathfrak I}^-(q)} \subset {\mathfrak I}^-(r)
\] 
which implies 
\[
 \overline{{\mathfrak I}^-(p)} \subset {\mathfrak I}^-(r),
\]
which in turn says that $p<r$ by applying Proposition \ref{2.1}. 
\end{proof}

\begin{lemma}
If $p < q$ then 
$
{\mathbb P}(p)^c \subset {\mathbb P}(q)^c
$, 
\end{lemma}
\begin{proof}
By Corollary \ref{cor:three}, $p<q$ implies $\mathfrak{I}^+(q) \subset \mathfrak{I}^+(p)$, therefore $\overline{\mathfrak{I}^+(q)} \subset\overline{\mathfrak{I}^+(p)}$. Then the lemma follows from the equality 
 $\overline{\mathfrak{I}^+(p)} \cup \overline{K}= \cap_{\pi \in {\mathbb P}(p)^c} \overline{H_\pi^+}$ 
 and the corresponding equality for $q$.
\end{proof}

We have the following result, analogous to Proposition \ref{2.1}:
\begin{proposition}\label{eq:stronger}
We have the equivalence:
\[
p<q\Leftrightarrow  \mathfrak{I}^+(p) \supset \overline{ \mathfrak{I}^+(q)}.
\]
\end{proposition}
\begin{proof}
First note that for $p<q$, we have \[
\left(\partial {\mathfrak I}^+(p )  \setminus K\right)  \cap \left(\partial {\mathfrak I}^+(q) \setminus K \right) = \emptyset.
\]
This follows from the fact that any element of $\partial {\mathfrak I}^+(q)  \setminus K$
is in the future of $p$. From this, the direct implication follows. The reverse implication is easy.
\end{proof}

Now we define the timelike Funk distance $F(p,q)$ on the subset $\Omega_{\leq}$ of $\Omega\times \Omega$.

\begin{definition}[The timelike Funk distance]\label{def:Funk10}
The function $F(p, q)$ on pairs of  
distinct points $p, q$ in $\Omega$ satisfying $p < q$ is given by the formula
\[
F(p, q) = \log \frac{d(p, b(p,q))}{d(q, b(p,q))}
\]
where $b(p,q)$ is the first intersection point of the ray $R(p,q)$ with $K$. Here,
$d(\cdot \, ,\cdot)$ denotes the Euclidean distance.

Note that the value of $F(p, q)$ is strictly positive. 

We extend the definition of $F(p, q)$ to the case where $p=q$, setting in this case $F(p, q) = 0$.
\end{definition}


Let $p$ and $q$ be two points in $\Omega$ such that $p<q$.  Let $\pi_0$ be a supporting hyperplane to $K$ at $b(p,q)$.  For $x$ in $\mathbb{R}^n$, let $\Pi_{\pi_0}(x)$ be the foot of the Euclidean perpendicular from the point $x$ onto that hyperplane. In other words, $\Pi_{\pi_0}: {\bf R}^n \rightarrow \pi_0$ is the Euclidean nearest point projection map. 
From the similarity of the Euclidean triangles  $\triangle (p, \Pi_{\pi_0}(p), b(p,q))$
and $\triangle (q, \Pi_{\pi_0}(q), b(p,q))$, we have
\[
\log \frac{d(p, b(p,q))}{d(q,  b(p,q))} = \log \frac{d(p, \pi_0)}{d(q,  \pi_0)}.
\]

Using the convexity of $K$, we now give a variational characterization of the quantity $F(p, q)$. 

For any unit vector $\xi$ in $\mathbb{R}^n$ and for any $\pi \in {\mathcal P}(p)$, we set
\[T(p,\xi, \pi)=\pi \cap \{p+t \xi  \,\, | \,\, t > 0\}\] 
if this intersection is non-empty.

For $p <q$ in $\mathbb{R}^n$, consider the vector $\xi = \xi_{pq}= \frac{q-p}{\Vert q-p\Vert}$ where the norm is the Euclidean one.  

 We then have $T(p,\xi_{pq}, \pi_{b(p,q)}) = b(p, q)\in R(p, q) \cap K$.
 
 In the case where $\pi\in\mathcal{P}(q)$ is not a supporting hyperplane of $K$ at $b(p,q)$, 
 the point $T(p, \xi_{pq}, \pi)$ lies outside $\overline{K}$ and,  again by the similarity of the Euclidean  triangles
$\triangle (p, \Pi_{\pi}(p), T(p, \xi_{pq}, \pi))$ and $\triangle (q, \Pi_{\pi}(q), T(p,\xi_{pq}, \pi))$, we get
\[
\frac{d(p, \pi)}{d(q, \pi)} =\frac{d(p, T(p,\xi_{pq},\pi))}{d(q, T(p,\xi_{pq}, \pi))}.
\]
Note  that as $\pi$ varies in ${\mathcal P}(q)$, the farthest point from $p$ on the ray $R(p,q )$ of the form $T(p,\xi_{pq}, \pi)$ is $b(p, q)$, and this occurs when $\pi$ supports $K$ at $b(p, q)$.
This in turn says that  a hyperplane $\pi_{b(p, q)}$ which supports $K$ at $b(p, q)$ minimizes the ratio 
\[
\frac{d(p, T(p,\xi_{pq}, \pi))}{d(q, T(p,\xi_{pq}, \pi))} 
\]
among all the elements of ${\mathcal P}(q)$.
Thus we obtain  
\begin{proposition}\label{prop:star}
 \[F(p,q) = \inf_{\pi \in {\mathcal P}(q)} \log \frac{d(p, \pi)}{d(q,   \pi)}.
 \] 
\end{proposition}

\begin{remark}  There is an analogous formula for the classical (non-timelike) Funk metric, where the infimum in the above formula is replaced by the supremum (see \cite{Y1} Theorem 1.)
\end{remark}

\begin{remark} \label{rem:future} The set $\mathfrak{I}^+(p)$ of future points of a point $p$, that is, the set of points $q$ satisfying $p<q$, reminds us of the cone of future points of some point $p$ in the ambient space of the physically possible trajectories of this point in the case of Minkowski space-time, that is, in the geometric setting of space-time for the theory of (special) relativity. The restriction of the distance function to the cone comes from the fact that a material particle travels at a speed which is less than the speed of light. The set of points on the rays starting at $p$ that are on the boundary $\partial \mathfrak{I}^+ (p)$ of the future region $\mathfrak{I}^+(p)$ becomes an analogue of the ``light cone" of  space-time (again using the language of relativity). In our definition of timelike geometry, the points in the light cone are excluded and we will postpone further discussion of light cones till \S \ref{s:anti}.
\end{remark}

We shall prove that the function $F(p,q)$ satisfies the reverse triangle inequality, which we call in this context, after Busemann, the \emph{time inequality}. This inequality holds  for triples of points $p, q$ and $r$ in $\Omega$, satisfying $p< q< r$:
\begin{proposition}[Time inequality]\label{time-ineq}
For any three points $p, q$ and $r$ in $\Omega$, satisfying $p< q< r$, we have
\[
F(p, q) + F(q, r) \leq F(p, r).
\]
\end{proposition}
\begin{proof}  
We use the formula given by Proposition \ref{prop:star} for the timelike Funk distance. We have, from  ${\mathcal P}(q) \supset {\mathcal P}(r)$ (Corollary \ref{p:inclusion}):
 \begin{eqnarray*}
F(p,q)+F(q,r)&=& 
\inf_{\pi \in {\mathcal{P}(q)}} \log \frac{ d(p, \pi)}{ d(q, \pi)} +
\inf_{\pi \in {\mathcal{P}(r)}} \log \frac{ d(q, \pi)}{ d(r, \pi)}\\
&\leq & 
  \inf_{\pi \in {\mathcal{P}(r)}} \log \frac{ d(p, \pi)}{ d(q, \pi)} +
\inf_{\pi \in {\mathcal{P}(r)}} \log \frac{ d(q, \pi)}{ d(r, \pi)}\\
&\leq &
\inf_{\pi \in {\mathcal{P}(r)}}\left( \log \frac{ d(p, \pi)}{ d(q, \pi)}
+  \log \frac{ d(q, \pi)}{ d(r, \pi)}
\right) \\
&=&
  \inf_{\pi \in {\mathcal{P}(r)}} \log \frac{ d(p, \pi)}{ d(r, \pi)}\\
  &=& F(p,r).
\end{eqnarray*}
\end{proof} 

 In the rest of this section, we study geodesics and spheres in timelike Funk geometries.  
  
First we consider geodesics for the timelike Funk distance. We start with the definition of a geodesic. This definition is the same as in an ordinary metric spaces, except that some care has to be taken so that the distances we need to deal with are always defined.

A \emph{geodesic} is a path $\sigma:J\to \Omega$, where $J$ may be an arbitrary interval of $\mathbb{R}$, such that for every pair $t_1\leq t_2$ in $J$ we have $\sigma(t_1)\leq \sigma(t_2)$ and for every triple $t_1\leq t_2\leq t_3$ in $J$ we have 
\[F(\sigma(t_1),\sigma(t_2))+ F(\sigma(t_2),\sigma(t_3))
=F(\sigma(t_1),\sigma(t_3)).\]

 It follows easily from the definition that for any $p<q$, the Euclidean segment $[p,q]$ joining $p$ to $q$ is the image of a geodesic. This means that the distance function $F$ satisfies Hilbert's Fourth Problem \cite{Pap-H} when this problem is generalized in an appropriate way to include timelike spaces. (We recall that one form of this problem asks for a characterization of metrics on subsets of Euclidean space such that the Euclidean lines are geodesics for this metric.)
In particular, the time inequality becomes an equality when $p$, $q$ and $r$ satisfying $p<q<r$ are collinear in the Euclidean sense.

It is important to note that whenever we use geodesics in timelike spaces, it is understood that these geodesics are equipped with a natural orientation. Traversed in the reverse sense, they are not geodesics.
 
 Let us make an observation which concerns the non-uniqueness of geodesics and the case of equality in the time inequality. Assume that the boundary of the convex hypersurface $K$ contains a Euclidean segment $s$. Take three points $p,q,r$ in $\Omega$ such that  $P(p,q)$ and $P(q,r)$ intersect $s$ (Figure \ref{segment}).
 \begin{figure}[!ht] 
\centering
 \psfrag{p}{\small p} 
 \psfrag{q}{\small q}
  \psfrag{r}{\small r}  
    \psfrag{s}{\small s}  
\includegraphics[width=0.45\linewidth]{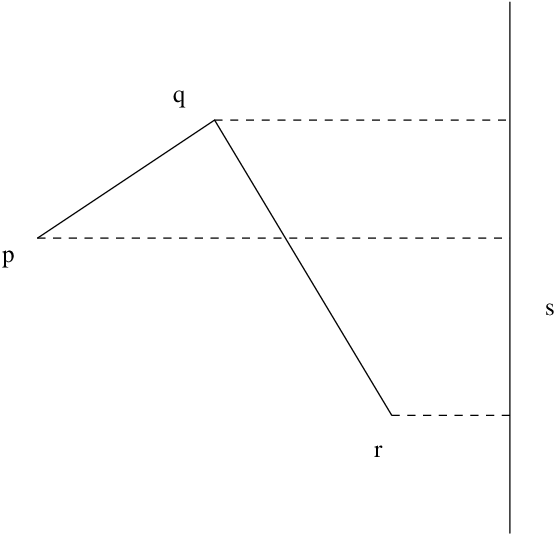}    \caption{\small {The broken segment $pqr$ is a geodesic}}   \label{segment}  
\end{figure}
Then, using the Euclidean intercept theorem, we have 
\[F(p,r)=F(p,q)+F(q,r).\]
Applying the same reasoning to an arbitrary ordered triple on the broken Euclidean segment $[p,q]\cup[q,r]$, we easily see that this segment is an $F$-geodesic.  More generally, by the same argument, we see that any arc in $\Omega$ monotonically heading toward the segment $s$ (that is, the Euclidean distance from a point on this arc and the segment $s$ is decreasing)
and such that any ray joining two consecutive points on the arc hits the segment $s$ is the image of an $F$-geodesic.

We deduce the following:   

\begin{proposition}\label{prop:geo1} A timelike Funk geometry $F$ defined on a set $\Omega_\leq $ associated with a convex hypersurface $K$ in $\mathbb{R}^n$ satisfies the following properties:
\begin{enumerate}
\item \label{eucl1} The Euclidean segments  in $\Omega$ that are of the form $[p,q]$, where $p<q$ are in $\Omega$, are $F$-geodesics. 
\item Let $p\in \Omega$ and  $b\in K(p)$. Then the semi-open Euclidean segment $[p,b[$ from  $p$ to $b$, equipped with the metric induced from the timelike Funk distance, is isometric to a Euclidean ray. 
\item The Euclidean segments in (\ref{eucl1}) are the unique $F$-geodesic segments
if and only if the convex set $I$ is strictly convex.
\end{enumerate}
\end{proposition}
 The proof is the same as that of the equivalence between (1) and (2) in Corollary 8.7 of \cite{PT1}, up to reversing some of the inequalities (i.e. replacing the triangle inequality by the time inequality), therefore we do not include it here.

After the geodesics, we consider spheres.
 \begin{definition}[Future spheres]
 At each point $p$ of $\Omega$, given a real number $r>0$, the \emph{future sphere} of radius $r$ centered at $p$  is the set of points in $\Omega$ that are in the future of $p$ and situated at $F$-distance $r$ from this point.
 \end{definition}

\begin{proposition}\label{prop:FS} At each point $p$ of $\Omega$ and for each $r>0$, the future sphere of center $p$ and radius $r$ is a piece of a convex hypersurface that is affinely equivalent to $K(p)$, the future of $p$ in $K$.
\end{proposition}

The proof is analogous to that of Proposition 8.11 of \cite{PT}, and we do not repeat it here.
 
 We next show a useful monotonicity result for a pair of timelike Funk geometries.  

  Given our open convex set $I$ with associated Funk distance $F$, we let
$\widehat{I} \supset I$ be another open convex set containing $I$ and $\widehat{F}(p, q)$ the associated timelike distance defined on the appropriate set of pairs $(p,q)$.

\begin{proposition}\label{prop:mono} For all $p$ and  $q$ in the domains of definition of both distances $F$ and $\widehat{F}$ (that is, for  pairs $(p,q)$ satisfying $p<q$ with respect to both convex sets $I$ and $\widehat{I}$), we have
\[
\widehat{F}(p, q) \geq F(p, q).
\]
\end{proposition}

\begin{proof} Using the notation of Definition \ref{def:Funk10}, we have
\[F(p, q) = \log \frac{d(p, b(p,q))}{d(q, b(p,q))}.\]
With similar notation, we have
\[\widehat{d}(p, q) = \log \frac{d(p, \widehat{b}(p,q))}{d(q, \widehat{b}(p,q))}.\]
Since  $\widehat{I} \supset I$, we have 
$d(p, \widehat{b}(p,q))=d(p, b(p,q))+x$ and $d(q, \widehat{b}(p,q))=d(q, b(p,q))+x$ for some $x \geq 0$.
The result follows from the fact that the function defined for $x \geq 0$ by
\[x\mapsto \frac{a-x}{b-x},\]
where  $b<a$ are two constants, is increasing.
 \end{proof}

\section{Timelike Minkowski spaces}\label{s:Minkowski}
Consider a finite-dimensional vector space, which we identify without loss of generality with $\mathbb{R}^n$. We introduce on this space  a \emph{timelike norm function} which we also call a \emph{timelike Minkowski functional}, in analogy with the usual Minkowski functional (or norm function) defined in the non-timelike case. To be more precise, we start with the following definition  (cf. \cite{B-Timelike} \S\,5).
\begin{definition}[Timelike Minkowski functional]\label{def:Min}
A \emph{timelike Minkowski functional} is a function $f$ satisfying the following:
\begin{enumerate}
\item $f$ is defined on $C\cup \{O\}$, where $C\subset\mathbb{R}^n$ is a proper open convex cone of apex the origin $O\in\mathbb{R}^n$, that is, an open convex subset invariant by the action of the positive reals $\mathbb{R}_{>0}$. (The fact that $C$ is proper means that it possesses a supporting hyperplane which intersects it only at the apex.)
\item $f(O)=0$.
 \item $f(x)> 0$ for all $x$ in $C$.
\item $f(\lambda x)=\lambda f(x)$ for all $x$ in $C$ and $\lambda>0$.
\item \label{c} $f\left( (1-t)x+ty\right) > (1-t)f(x)+t f(y)$ for all $x, y \in C$ and for all $0<t<1$. \label{concavity}
\end{enumerate}
\end{definition}
We shall say that $C\subset\mathbb{R}^n$ is the \emph{cone associated with the timelike Minkowski functional $f$}.

Note that since $-f$ is a convex function, it is continuous.

The \emph{unit sphere} $B$ of such a timelike norm function $f$ is the set of vectors $x$ in $C$ satisfying $f(x)=1$. In general, $B$ is a piece of a hypersurface in $\mathbb{R}^n$ which is concave when viewed from the origin $O$ (see Figure \ref{concave}).  Our definition allows the possibility that $B$ is asymptotic to the boundary of the cone $C$. The unit sphere $B$ is called the \emph{indicatrix} of $f$. 

\begin{figure}[!ht] 
\centering
 \psfrag{O}{\small  $O$} 
 \psfrag{B}{\small $B$} 
\includegraphics[width=0.55\linewidth]{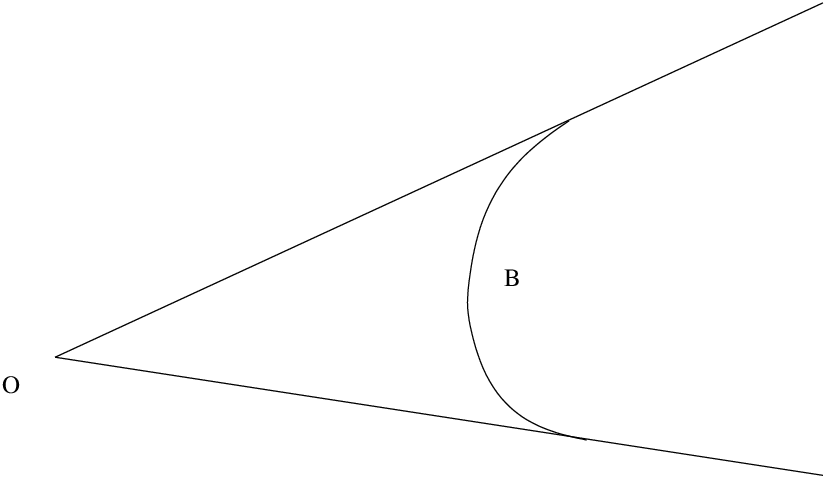}    \caption{\small {The indicatrix $B$ in the tangent space to a point in $\Omega$.}}   \label{concave}  
\end{figure}

The reason of the adjective \emph{timelike} in the above definition is that in the Lorentzian setting, the Minkowski norm measures the lengths of vectors in the time cone, which is the part of space-time where material particles move. In particular, there is a timelike Minkowski functional $f$ for the standard Minkowski space-time $
{\mathbb R}^{3,1}$, equipped with the Minkowski metric
\[
ds^2 = - dx_0^2 + d x_1^2 + dx_2^2 + dx_3^2.
\] 
 It is given by
\[
f(x) = \sqrt{-(- x_0^2 + x_1^2 + x_2^2 + x_3^2)}
\]
and it is defined
for vectors $x$ in ${\mathbb R}^4$ satisfying $ - x_0^2 + x_1^2 + x_2^2 + x_3^2 < 0$ or $x=0$.
 
 \section{Timelike Finsler structures}\label{s:Timelike-Finsler}
   
 \begin{definition}[Timelike Finsler structure]\label{def:TFS} A timelike Finsler structure on a differentiable manifold $M$ is a family of functions $\{f_p\}_{p\in M}$, where for each $p\in M$, $f_p$ is a timelike Minkowski functional defined on the tangent space $T_pM$ of $M$ at $p$. In particular, in the tangent space at each point $p$ in $M$, there is a cone $C_p$ associated  to $f_p$ which plays the role of the cone $C\cup \{O\}$ associated in Definition \ref{def:Min} to a general timelike Minkowski functional. We assume that $f_p$ together with its associated cone $C_p$ depend continuously on the point $p$. \end{definition}

 We shall sometimes use the notation $f(p,x)$ instead of $f_p(x)$ for the Minkowski functional.
 
 In the situations considered in this paper, the differentiable manifold $M$ will be generally an open subset of either a Euclidean space $\mathbb{R}^n$ or a sphere $S^n$. (In some rare cases, it will be a subset of a hyperbolic space $\mathbb{H}^n$.)

We say that a piecewise  $C^1$ curve  $\sigma: [0,1]\to  M$, $t\mapsto \sigma(t)$ is {\it timelike} if at each 
time $t\in J$ the tangent vector $\sigma'(t)$ (if this tangent vector exists) is an element of the cone $C_{\sigma(t)} \subset T_{\sigma(t)}M$. At the points $t\in J$ where $\sigma$ is not $C^1$, there are two naturally defined tangent vectors and we assume that both are in $C_{\sigma(t)}$.

\begin{definition}[The partial order relation]\label{prec-order}
If $p$ and $q$ are two points in $M$, we write $p \prec q$, and we say that \emph{$q$ is in the $\prec$-future of $p$}, if there exists a timelike piecewise $C^1$ curve  $\sigma: [0,1] \to M$ joining $p$ to $q$.
\end{definition}

 We define the {\it length} of a piecewise $C^1$ timelike curve $\sigma: [0,1] \to M$ by the Lebesgue integral 
\[
{\rm Length}(\sigma) = \int_0^1 f_{\sigma(t)}(\sigma'(t)) \, dt.
\]

We then define a function $\delta$ on pairs of points $(p, q)$ satisfying  $p\prec q$ by setting
\begin{equation}\label{eq:sup}
\delta(p, q) = \sup_{\sigma} {\rm Length}(\sigma)
\end{equation}
where the supremum is taken over all the timelike piecewise $C^1$ curves $\sigma:[0,1]\to M$ satisfying $\sigma(0)=p$ and $\sigma(1)=q$. We shall show that $\delta$ defines a timelike distance function.

 It is easy to see from the definition of $\delta$ that it satisfies the timelike inequality, once we show the following
 \begin{lemma}\label{lemma:finite}
 For any pair $p$ and $q$ satisfying $p\prec q$, we have   $\delta(p, q) <\infty$.
 \end{lemma}

\begin{proof}
To see that the supremum in (\ref{eq:sup}) is finite, we introduce a reference metric on a chart of the manifold modelled on the Minkowski space-time $(\mathbb{R}^n, -c^2 dt^2 + dx_1^2 + \dots + dx_{n-1}^2)$ as follows.  Let $(U, \phi)$ be a local chart on $M$ containing a point $p$ so that $\phi(U)$ is an open subset of $\mathbb{R}^{n}$ with $\phi(p)  = 0$.  As $\phi: U \rightarrow \phi(U)$ is a diffeomorphism, each open cone $C_p$  in $T_p M$ on which the Minkowski functional $f_p: C_p \rightarrow \mathbb{R}$ is defined is mapped onto a proper convex cone $C_{\phi(p)}$ in $\mathbb{R}^n$ by the linear map $d\phi_p$.
Hence we have a field of proper cones $\{C_x\}_{x \in \phi(U)}$.  By the continuity of $d \phi_p$ in $p$, there exists an open neighborhood  $V \subset U$ of $p$ so that on $\overline{V}$ there is a field of supporting hypersurfaces of $\{C_x\}_{x \in \overline{V}}$: $\{\pi_x  \subset T_x \mathbb{R}^n\}$ with all the  hyperplanes $\{\pi_x\}_{x \in \overline{V}}$ sharing the same normal vector in $\mathbb{R}^n$. 

We now introduce a  Minkowski metric $g_c = -c^2 dt^2 + dx_1^2 + \dots + dx_{n-1}^2$  on $\overline{V} \subset \mathbb{R}^n$ where the constant $c$ (the ``speed of light") will be determined below. The $x_1x_2\dots x_{n-1}$-plane is identified 
with $\pi_x$ for each $x \in \overline{V}$. We also consider $B_1(x) \subset T_x \mathbb{R}^n$, the set of future directed timelike vectors $v$ with $-1 < 
g_c(v, v) \leq 0$.   Then we can choose the constant $c>0$  sufficiently large so that 
\begin{enumerate}
\item the light cone $\{v \in T_x \mathbb{R}^n \,|\, g_c(v, v) < 0\}$ properly contains $C_x$ at each $x \in \overline{V}$;
\item each $g_c$-unit vector $v$ in $\partial B_1(x) \cap C_x$, which is identified with a tangent vector $w= (d \phi_x)^{-1}v$ in $T_{\phi^{-1} (x)} M$, has norm  $f_{\phi^{-1} (x)}(w) < 1$. 
\end{enumerate}
So far, we have defined an auxiliary norm $f^M_q:C_q \rightarrow  \mathbb{R}$ for any $q \in \phi^{-1} 
(\overline{V})$ with $f^M_q(v) > f_q(v)$. We denote the distance with respect to the Minkowski metric $g_c$ by $d_c$.  Note that the condition $(1)$ ensures that a timelike curve in $M$ with respect to the family of norms $f_q$ is also timelike for the auxiliary family of norms $f_q^M$.  

Now given a  timelike $C^1$-curve $\sigma: [a, b] \rightarrow \phi^{-1}(\overline{V}) \subset M$ through $p = \sigma(0)$, we have the following length comparison 
\[
  \int_a^b f_{\sigma(t)} (\sigma'(t)) \, dt  < \int_a^b f_{\sigma(t)}^M(\sigma'(t)) \, dt
\] 
with the auxiliary length bounded above,
\[
\int_a^b f_{\sigma(t)}^M(\sigma'(t)) \, dt < d_c (\phi(\sigma(a)), \phi(\sigma(b))) < \infty,
\]
as the line segment $[\phi(\sigma(a)), \phi(\sigma(b))]$ is the length maximizing timelike curve in the Minkowski space-time
 $(\mathbb{R}^n, -c^2 dt^2 + dx_1^2 + \dots + dx_{n-1}^2)$.
 \end{proof}

It follows that the timelike distance function $\delta$ defines a timelike structure on the space $M_{\preceq}$ of pairs of points $(x,y)$ in $M\times M$ satisfying $x\preceq y$.  This timelike structure is the analogue of the so-called \emph{intrinsic metric} in the classical (non-timelike) case. We call $\delta$ the \emph{timelike intrinsic distance} associated with the timelike Finsler structure.

\section{The timelike Finsler structure of the timelike Funk distance}\label{s:Funk-Finsler}

In this section, we show that the timelike Funk distance $F$ associated with a convex hypersurface $K$ in $\mathbb{R}^n$ is timelike Finsler in the sense defined in \S \ref{s:Timelike-Finsler}. In other words, we show that on the tangent space at each point of $\Omega=\mathbb{R}^n\setminus \overline{K}$, there is a  timelike Minkowski functional (which we also call a timelike norm) which makes this space a timelike Minkowski space, such that the timelike Funk distance $F(p,q)$ between two points $p$ and $q$ is obtained by integrating this timelike norm on tangent vectors along piecewise $C^1$ paths joining $p$ to $q$ and taking the supremum of the lengths of such paths.  The paths considered are restricted to those where the tangent vector at each point of $\Omega$ belongs to the domain of the timelike Minkowski functional.

For every point $p$ in  $\Omega_{\leq}$ associated with a convex hypersurface $K$, there is a \emph{timelike Minkowski functional} $f_F(p,v)$ defined on the subset of the tangent space $T_p\Omega$ of $\Omega$ at $p$ consisting of the non-zero vectors $v$ satisfying 
\[
 p + tv \in \mathfrak{I}^+(p) \mbox{ for some } t > 0
\]
where, as before, $\mathfrak{I}^+(p)$ is the future of $p$.
 We denote by  $C^+(p) \subset T_p \Omega$ the set of   vectors $v$ that satisfy this property or are the zero vector.  We define the function $f_F(p,v)$ for  $p \in \Omega$ and  $v \in C^+(p)$  by the following formula:
\begin{equation}\label{eq:P2TF}
f_F(p, v) = \inf_{\pi \in {\mathcal P}(p)} \frac{\langle v, \eta_\pi \rangle }{d(p, \pi)}
\end{equation}
for $v \in C^+(p)$   where ${\mathcal P}(p)$ is as before the set of supporting hyperplanes to $K$ separating $p$ from the interior of $K$ and where for each hyperplane $\pi$ in ${\mathcal P}(p)$, $\eta_\pi$ is the unit tangent vector at $p$ perpendicular to $\pi$ and pointing toward $\pi$. We also define $f_F(p,0)=0$ for all $p$ in $\Omega$. We shall show that this defines a timelike Minkowski functional and that this functional is associated with a timelike Finsler geometry underlying the timelike Funk distance $F$.

By elementary geometric arguments (see \cite{Y1} for a detailed discussion in the non-timelike case which can be adapted to the present setting) it is shown that
\begin{equation}\label{eq:P1}
f_F(p, v) = \frac{\|v\|}{\inf\{t \,\, | \,\, p + t \frac{v}{\|v\|} \in  \overline{K} \}}
\end{equation}
for any nonzero vector $v \in C^+(p)$. 

Note that the quantity $\inf\{t \,\, | \,\, p + t \frac{v}{\|v\|} \in \overline{K} \}$ in the denominator is the Euclidean length of the line segment 
from $p$ to the point where the ray $p + tv$ hits the convex set $\overline{K}$ for the first time. A simpler way to write the functional defined in (\ref{eq:P2TF}) is:
 \begin{equation}\label{eq:P211}
 f_F(p,v)=\sup\{t:p+v/t\in \overline{K}\}.
 \end{equation}

We have the following;

\begin{proposition}
The functional $f_F(p, v)$ defined on the open cone $C^+(p)$ in $T_p \Omega \cong {\mathbb R}^n$ satisfies all the properties required by a timelike Minkowski functional.  
\end{proposition}

\begin{proof}
It is easy to check the properties required by Definition \ref{def:Min}.  Note that the last property in this definition, namely, the concavity of the linear functional $f_F$ on the tangent space $T_p \mathbb{R}^n$, follows from the fact that $f_F$ is an infimum over $\mathcal{P}(p)$ of linear (and in particular concave) functionals, and such an infimum is concave.   \end{proof}

As in \S \ref{s:Timelike-Finsler}, we say that a piecewise  $C^1$ curve  $\sigma: [0,1]\to   \Omega$, $t\mapsto \sigma(t)$, defined on an interval $[0,1]$ of $\mathbb{R}$, is {\it timelike} if for each $t\in [0,1]$ the tangent vector $\sigma'(t)$ is an element of the cone $C^+(\sigma(t)) \subset T_{\sigma(t)}\Omega$.

If $p$ and $q$ are two points in $\Omega$, we write $p \prec q$, and we say that \emph{$q$ is in the $\prec$-future of $p$}, if there exists a timelike piecewise $C^1$ curve  $\sigma: [0,1] \to   \Omega$ joining $p$ to $q$.

\begin{proposition}\label{order-equiv}
The two partial order relations $<$ and $\prec$ defined on $\Omega$ coincide; namely, for any two points $p$ and $q$ in $\Omega$, we have 
\[
p < q \Leftrightarrow p \prec q.
\]
\end{proposition} 

\begin{proof}
The implication $p < q \Rightarrow p \prec q$ follows from that fact that for $p < q$, the parameterized curve 
\[
\sigma(t) = p + \frac{q-p}{\|q-p\|}t  
\]
for $t \in [0,1]$ is a $C^1$ timelike curve from $p$ to $q$.  

The reverse implication, $p < q \Leftarrow p \prec q$, follows from the following lemma:

\begin{lemma} \label{lem:rev}
Assume that $p$ and $q$ in $\Omega$ satisfy $p<q$. Then, for an arbitrary piecewise $C^1$ timelike curve $\sigma:[0,1]\to \Omega$ with $\sigma(0) =p$ and $\sigma(1)=q$, we have
\[
\sigma(t) \in \mathfrak{I}^+(p)
\]
for all $t$ in $[0,1]$. 
\end{lemma}
\begin{proof}[Proof of Lemma \ref{lem:rev}]

The path $\sigma$, being timelike, starts at the point $p$ with a right derivative at $p$ pointing strictly inside the cone $C(p)$. This implies that the point $\sigma(t)$ is strictly inside the set ${\mathfrak I}^+ (p)$ for any sufficiently small  $t$. 

If the image of $\sigma$ is not completely contained in $\mathfrak{I}^+(p)$, then there is a smallest $t_0>0$ in $[0,1]$ such that $\sigma(t_0)$ is on the boundary of set ${\mathfrak I}^+ (p)$. Let $p_0=\sigma(t_0)$ and assume $\sigma$ is differentiable at $t_0$. Then the tangent vector to $f(\sigma)$ at $t_0$ is a vector at $p_0$ which is either contained in the boundary of ${\mathfrak I}^+ (p)$ or points outside this ${\mathfrak I}^+ (p)$. But this contradicts the fact that the tangent vector to $f(\sigma)$ at $p_0$ is in $C_{p_{0}}$. 

If $\sigma$ is not differentiable at $t_0$, then (since this curve is piecewise $C^1$) there are two tangent vectors at this point, and the same argument applied to one of these vectors gives the same contradiction. 

Thus, the image of $\sigma$ is completely contained in $\mathfrak{I}^+(p)$, which proves the lemma.

%
\end{proof}

We continue the proof of Proposition \ref{order-equiv}. For $p$ and $q$ in $\Omega$ satisfying $p \prec q$, we let $\sigma:[0,1]\to \Omega$ be an arbitrary piecewise $C^1$ timelike curve satisfying  $\sigma(0) =p$ and $\sigma(1)=q$. From Lemma  \ref{lem:rev}, we have $\sigma(1) = q  \in {\mathfrak I}^+ (p)$, therefore $p < q$. This finishes the proof of Proposition \ref{order-equiv}. 
\end{proof}

As we did in \S \ref{s:Timelike-Finsler}, we denote by $\delta$ the timelike intrinsic distance function associated with this timelike Finsler structure, namely,
\begin{equation}
\delta(p, q) = \sup_{\sigma} {\rm Length}(\sigma)
\end{equation}
where the supremum is taken over all the timelike piecewise $C^1$ curves $\sigma:[0,1]\to \Omega$ satisfying $\sigma(0)=p$ and $\sigma(1)=q$.  Like in Lemma \ref{lemma:finite}, it is seen that the intrinsic distance $\delta(p, q)$ for $p<q$ is finite. 

Thus, the domain of definition of the set $\Omega_<$ associated with the partial order $<$ for  the timelike Funk distance $F_1^2$  and the domain of definition $\Omega_\prec$ for the  timelike Finsler distance function $\delta_1^2$  coincide. Furthermore, we shall prove the equality $\delta(p, q) =F (p, q)$ for any pair $p < q$ in $\Omega_<$. We state this as follows:

\begin{theorem}\label{th:distance}
The value of the timelike Finsler distance $\delta(p, q)$ for a pair $(p,q)\in \Omega_{\leq}$ coincides with $F(p, q)$. That is, we have
\[
F(p, q) = \delta(p, q).
\]
\end{theorem}

In other words, we have the following
\begin{theorem}\label{Funk-Finsler-1} 
The timelike Funk geometry is a timelike Finsler structure defined by the 
Minkowski functional $f_F(p,v)$.
\end{theorem}

The timelike Minkowski functional $f_F(p,v)$ which underlies a timelike Funk geometry has a property which makes that metric the
\emph{tautological Finsler structure} associated with the hypersurface $K$ (or the convex body $I$).  
The term ``tautological" is due to the fact that the indicatrix of the timelike Minkowski functional at $p \in \Omega$, that is, the set 
\[
{\rm Ind}(p) = \{v \in C^+(p) \subset T_p \Omega \,\, | \,\, f_F(p, v) = 1\},
\]
is affinely equivalent to the part of $K$ that is ``visible from the point $p$", that is, the relative interior (with respect to the topology of $K$) of the intersection of that hypersurface with $\overline{\mathfrak{I}^+(p)}$, the closure in $\mathbb{R}^n$ of the subset $\mathfrak{I}^+(p)$.    This property is the timelike analogue of a property of the indicatrix of the classical Funk metric which makes it tautological, as noticed in the paper \cite{PT}.

We also note that with this identification, given a pair of points $p, q$ with $p < q$, there always exists a distance-realizing (length-maximizing) $F$-geodesic from $p$ to $q$, namely, the Euclidean segment $[p,q]$.

\begin{proof}[Proof of Theorem \ref{th:distance}] For any pair $(p, q)\in(\Omega\times\Omega)$ satisfying $p<q$, we consider the map
\begin{equation}\label{eq:sigma}
\sigma:[0,1]\to\mathbb{R}^n
\end{equation}
parametrizing the Euclidean segment $[p,q]$ proportionally to arc-length $t$ with $\sigma(0)=p, \sigma(1)=q.$  Then we have 
\[
\int_0^1 f_F(\sigma(t), \sigma'(t)) \, dt = \log \frac{d(p, b(p,q))}{d(q, b(p,q))} = F(p, q),
\] 
since 
\[
\frac{d}{dt} \log \frac{d(p, b(p,q))}{d(\sigma(t), b(p,q))} = f_F(\sigma(t), \sigma'(t)).
\]
By taking the supremum over the set of all paths from $p$ to $q$, we obtain the inequality
\begin{equation}\label{eq:11}
\delta(p, q) \geq  F(p, q).
\end{equation}

Before continuing the proof of Theorem  \ref{th:distance}, we show a useful monotonicity property of the intrinsic distance.

Let $\widehat{I} \supset I$ be an open convex set containing the convex subset $I$ of \S \ref{Funk} and let $\widehat{K}$ be its bounding hypersurface. Let $\widehat{F}$ be the associated timelike Funk metric, $f_{\widehat{F}}$  its associated timelike Minkowski functional, and $\widehat{\delta}$ the associated intrinsic distance. (Note that the domains of definition of   $f_{\widehat{F}}$ and $\widehat{\delta}$ contain those of $f_F$ and $\delta$ respectively.) We have the following:

\begin{lemma}\label{lem:another1} For $(p,q)$ in the domains of definition of both intrinsic distances $\delta$ and $\widehat{\delta}$, we have
\[
\widehat{\delta}(p, q) \geq \delta(p, q).
\]
\end{lemma}

\begin{proof} 
The timelike Minkowski functionals   $f_F$ and $f_{\widehat{F}}$ satisfy the following inequality  
\[
f_{\widehat{F}}(x, v) \geq f_F(x, v) 
\]
whenever the quantities involved are defined concurrently.  This follows from the definition of the Minkowski functional:
\begin{equation}\label{eq:P22}
f_F(p, v) = \frac{\|v\|}{\inf\{t \,\, | \,\, p + t \frac{v}{\|v\|} \in  K \}}
\end{equation}
for any nonzero vector $v$ in both domains of definition, as $\widehat{K}$ is closer to $p$ than $K$.  
Hence by integrating each functional along an  admissible path (note that admissible paths for $\delta$ are also admissible paths for $\widehat{\delta}$) and taking the supremum over these paths, we obtain
\[
\widehat{\delta}(p, q) \geq \delta(p, q).
\]
 \end{proof}

\noindent {\it Proof of Theorem  \ref{th:distance} continued}.--- Suppose that we have a convex hypersurface $K$ bounding an open convex set $I $, and for each $(p, q) \in \Omega_<$,  let 
\[\widehat{I} = H^+_{\pi_{b(p, q)}},\]
 where $H^+_{\pi_{b(p, q)}}$  is 
the open half-space  bounded by a hyperplane $\pi_{b(p, q)}$ supporting $K$ at $b(p, q)$ and containing $I$. The open set $\widehat{I}$ is equipped with its intrinsic distance $\widehat{\delta}$. Applying Lemma \ref{lem:another1} to this setting where a convex set  $\widehat{I}$  
contains $I$, we obtain $\widehat{\delta} \geq \delta$.

For the open half-space $\widehat{I} = H^+_{\pi_{b(p, q)}}$, the values of $F(p, q)$, $\widehat{F}(p, q)$ and $\widehat{\delta}(p, q)$ all coincide. Indeed, under the hypothesis $\widehat{I} = H^+_{\pi_{b(p, q)}}$, the set $
\mathcal P$ of supporting hyperplanes consists of the single element $\pi_{b(p, q)}$, and the line segment
$\sigma$ from $p$ to $q$ defined in (\ref{eq:sigma}) is a length-maximizing path, since every timelike path for $\widehat{I}$ is $\widehat{F}$-geodesic.  

Combining the above observations, we have 
\begin{equation}\label{eq:12}
 F(p, q) = \widehat{F}(p, q) = \widehat{\delta}(p, q) \geq \delta(p, q) \geq F(p, q)
\end{equation}
and  the equality $\delta(p, q) =F(p, q)$ follows.
\end{proof}

   We end this section by the following convexity property on the timelike Funk distance associated with a strictly convex hypersurface $K$:
   \begin{theorem} Assume that $K$ is strictly convex. 
For any point $x$ in $\Omega$ and $M > 0$, the set of points 
\[
S_M(x) = \{p \in \Omega \,\, | \,\, p < x \mbox{ and } F(p, x) > M \}
\]
is a convex subset of $\Omega = {\mathbb R}^n \setminus \overline{K}$.
\end{theorem}
\begin{proof} Since $K$ is strictly convex, any $F$-geodesic is a Euclidean segment.
Given $p_1$ and $p_2$ in $S_K(x)$, parameterize the Euclidean segment $[p_1,p_2]$  with an affine parameter $t \in [0,1]$ by $s(t)$, with $s(0)=p_1$ and $s(1)=p_2$. We shall show that the function $t\mapsto F(s(t),x)$ is concave.

By Proposition \ref{prop:star}, we have
\[
F(s(t), x) = \inf_{\pi \in \mathcal{P}(x)} \log \frac{d(s(t), \pi)}{d(x, \pi)}.
\]
  Fix a supporting hyperplane $\pi$ in $\mathcal{P}$. Then \[
\frac{d}{dt}  \log \frac{d(s(t), \pi)}{d(x, \pi)}=  \frac{ \langle -\nu_\pi (s(t)), \dot{s}(t) \rangle}{ d(s(t), \pi ) } 
\]
and
\[
\frac{d^2}{dt^2} \log \frac{d(s(t), \pi)}{d(x, \pi)} = - \frac{\langle -\nu_\pi(s(t)), \dot{s}(t) \rangle^2}{d(s(t), \pi)^2}  \leq 0,
\]
where $\nu_\pi(x)$ is the unit vector 
at $x$ perpendicular to the hypersurface $\pi$ oriented toward $\pi$.  In particular $- \nu_\pi$ is the gradient vector of the function $d(x, \pi)$. 
The sign of the second derivative says that $\log \frac{d(s(t), \pi)}{d(x, \pi)}$ is concave in $t$ for each $\pi \in \mathcal{P}$.  By taking the infimum over $
\pi \in \mathcal{P}$, the resulting function $F(s(t),x)$ is concave in $t$. 

This implies that the super-level set $S_K(x)$ of the Funk distance $F(\cdot, x)$ is convex. 
\end{proof}

As an analogous situation in special relativity, the super-level set of the past-directed temporal distance measured from a fixed point in the Minkowski space-time ${\mathbb R}^{n, 1}$ is convex.  For example, the set below the past-directed hyperboloid: $S_1(0)= \{(x_0, x_1) \in {\mathbb R}^{1,1} \,\, | \,\, -x_0^2 + x_1^2 < -1, \ x_0<0 \}$ is convex.

\section{The timelike Euclidean relative Funk geometry}\label{s:ETRF}

Let $K_1$ and $K_2$ be two disjoint convex hypersurfaces in $\mathbb{R}^n$ that bound convex sets $I_1$ and $I_2$ respectively, with $\overline{K_1}= K_1\cup I_1$ and  $\overline{K_2}= K_2\cup I_2$ being the closures of $I_1$ and $I_2$ respectively.

 A \emph{timelike Euclidean relative Funk geometry} is associated with the ordered pair $K_1,K_2$. Its 
 underlying space is the subset $\Omega$ of $\mathbb{R}^n$, as pictured in Figure \ref{fig3}, defined as the union
 \[\Omega=\cup]a_1,a_2[,\]
the union being over the intervals $]a_1,a_2[ \subset\mathbb{R}^n$  such that $a_1\in K_1$, $a_2\in K_2$, $]a_1,a_2[\cap (K_1\cup K_2)=\emptyset$  and such that there is no supporting hyperplane $\pi$ to $K_1$ or to $K_2$ containing $]a_1,a_2[$.

 We let $K_1^2\subset K_2$ be the set of points $k_2\in K_2$ such that there exists a point $k_1\in K_1$ with $]k_1,k_2[\cap (K_1\cup K_2)=\emptyset$. We shall say that $K_1^2$ is the subset of $K_2$ \emph{facing $K_1$}.

\begin{figure}[!ht] 
\centering
   \psfrag{A}{\small The space $\Omega$ of the relative Funk metric}
\includegraphics[width=2in]{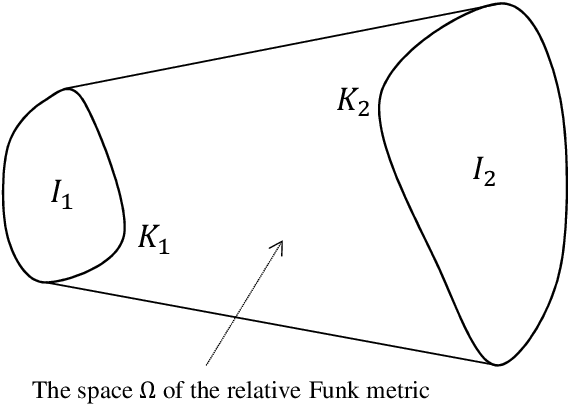}    
\caption{\small{The space underlying the relative Funk metric }}
\label{fig3}  
\end{figure}

In the rest of this section, the pair $K_1,K_2$ is always understood to be an ordered pair, even if the notation we use does not reflect this fact. For reasons that will become apparent soon, $K_1$ represents the past, and $K_2$ the future. We shall also say that $K_2$ is the future of $K_1$.

\begin{definition}[Order relation and relative future] \label{def:o} With  the above notation, for $p$ and $q$ in $\Omega$, we write $p<q$ if there exists an open Euclidean segment $]a_1,a_2[\subset \Omega$   with $a_i \in K_i$ such that the four points $a_1,p,q,a_2$ are collinear in that order, and such that $]a_1,a_2[$ is not contained in any supporting hyperplane of $K_1$ or of $K_2$. 

If $p<q$ then we say that  $q$ \emph{lies in the future} of $p$, and that $p$ \emph{lies in the past} of $q$.

\end{definition}

 We write $p\leq q$ if either $p<q$ or $p=q$.

  We denote by $\Omega_<$ (resp. $\Omega_{\leq}$) the set of ordered pairs $(p,q)$ in $\Omega\times\Omega$ 
satisfying $p<q$  (resp. $p\leq q$).   The set  $\Omega_<$ is disjoint from the diagonal set $\{(x, x) \,\,| \,\, x \in 
\Omega\} \subset \Omega\times\Omega$.

\begin{figure}[!ht] 
\centering
 \psfrag{p}{\small $p$}
   \psfrag{A}{\small The relative future of $p$}
    \psfrag{K1}{\small $K_1$}
     \psfrag{K2}{\small $K_2$}
\includegraphics[width=2in]{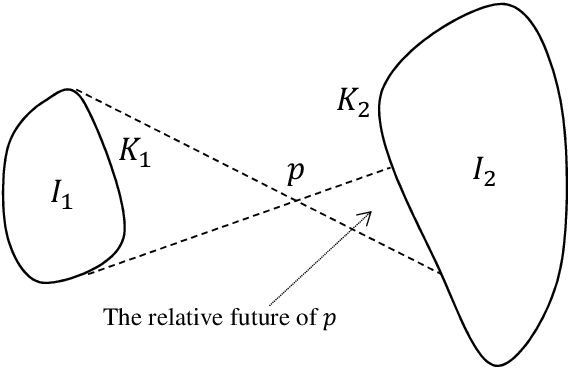}    \caption{\small {Relative future of $p$}}   \label{fig4}  
\end{figure}

  For every point $p$ in $\Omega$, its relative future set $\mathfrak{I}_2^+(p)$ is nonempty, open and connected. Note that this set  also depends on $K_1$ even though we do not include this information in the notation in order to make it lighter.
 
We shall sometimes use the word ``future" instead of the expression ``relative future" if the context is clear.

\begin{definition}[The relative future in $K_2$ of a point] For $p$  in $\Omega$, we consider the following subset  of $K_2$:  \[K_1^2(p)=\{a_2\in K_2  \ \mathrm{ such \  that }\   \exists a_1\in K_1 \ \mathrm{ with } \ p\in ]a_1,a_2[\subset \Omega\}\]
and we say that $K_1^2(p)$ is the \emph{relative future of $p$ in $K_2$}.

\end{definition}
\begin{definition}
 For $p$ in $\Omega$, we denote by  $\mathfrak{I}_2^+(p)$ the set of all points $q\in \Omega$ which are in the relative future of $p$, and we call this set the \emph{relative future of $p$}.
 \end{definition}
The \emph{relative future} of $p$ is represented in Figure \ref{fig4}.

\begin{figure}[!ht] 
\centering
 \psfrag{p}{\small $p$} 
   \psfrag{A}{\small $\widetilde{K_2}$}
    \psfrag{B}{\small $\widetilde{I_2}$}
    \psfrag{K1}{\small $K_1$}
     \psfrag{K2}{\small $K_2$}
\includegraphics[width=2 in]{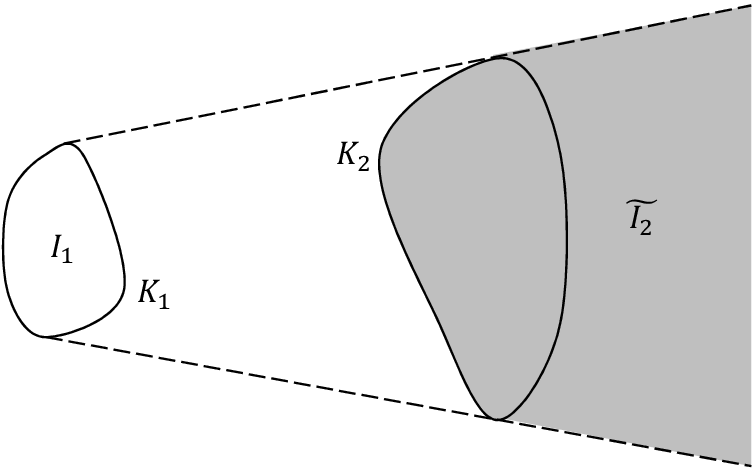}      \caption{\small{$\widetilde{K_2}$ is the boundary of the closure of $\widetilde{I_2}$}}    \label{fig1}  
\end{figure}
 
 In order to formulate the bases of the relative Funk geometry, we introduce the following notation.

 \medskip
 $\bullet$   $\widetilde{\mathcal{P}_2}$  is the set of supporting hyperplanes to $K_2$ at points in $K_1^2$. 

\medskip
 $\bullet$ $^iH_\pi^+$ is the open half-space,  bounded by a hyperplane $\pi$, and containing the convex set $I_i$.
 $^iH_\pi^-$ is the complementary half-space, namely, the open half-space bounded by $\pi$, not containing $I_i$. 

 \medskip
 $\bullet$ $\widetilde{I_2}=\cap {^2H_{\pi}^+}$
where $\pi$ varies in  $\widetilde{\mathcal{P}_2}$.
This is an open convex subset of $\mathbb{R}^n$ and it contains $I_2 = \cap {^2H_{\pi}^+}$ where the union is over $\pi$ varying in $\mathcal{P}_2$.

 \medskip
 $\bullet$  $\widetilde{K_2}$ is the boundary of the closure of $\widetilde{I_2}$. 
 ($\widetilde{I_2}$ are represented in Figure \ref{fig1}.)

 \medskip
$\bullet$  $\overline{\widetilde{K_2}}= \widetilde{K_2}\cup  \widetilde{I_2}$.

 \medskip
 $\bullet$  $\widetilde{\mathbb{P}_2}(p)$ is the set of  hyperplanes in $\mathbb{R}^n$ separating $p$ from $\widetilde{I_2}$.

For every element $\pi  \in \widetilde{\mathbb{P}_2}$,  $^2H^+_\pi$ is the open half-space bounded by the hyperplane $\pi$ and containing $\widetilde{I_2}$, and $^2H^-_\pi$ the open half-space bounded by $\pi$ and not containing $\widetilde{I_2}$. 
We have:
\[
\widetilde{I_2}= \cap_{\pi \in {\widetilde{\mathcal{P}_2}}}  {^2H^+_\pi} = \cap_{\pi \in {\widetilde{\mathbb{P}_2}}}  {^2H^+_\pi}
\]
 
and
\begin{equation}\label{eq:P2}
 \widetilde{\mathbb{P}_2}(p) = \{\pi \in { \widetilde{\mathbb{P}_2}} \,\, | \,\, p \in {^2H^-_\pi} \}.
\end{equation} 

 \begin{definition}[The relative past of a point]\label{def:rp}
 For $p\in \Omega$, the \emph{relative past of} $p$, denoted by ${\mathfrak I_2}^-(p)$,  is the set of points $q$ in $\Omega$ such that $p$ is in the relative future of $q$.
 \end{definition}
 
   The set ${\mathfrak I}_2^+(p)$ is an open subset of $\mathbb{R}^n$. It is characterized by the following: 
\begin{equation}\label{char:I}
{\mathfrak I}_2^+(p) = \mathrm{Int}\left(\{(\cap_{\pi \in \widetilde{\mathbb{P}_2}^c(p)} \overline{^2H^+_\pi}) \setminus \overline{\widetilde{K_2}} \} \cap   \{\cap_{\pi \in \widetilde{\mathbb{P}_1}(p)} \overline{^1H^-_\pi} \} \right)
\end{equation}
 where, as before, $\mathrm{Int}( \quad )$ denotes the interior of a set.  We recall that we are using the notation $$\mathrm{Int} \{ \cap_{\pi \in \widetilde{\mathbb{P}_2}(p)^c} \overline{^2H^+_\pi}) \setminus \overline{\widetilde{K_2}} \}$$ for the future set for the non-relative Funk geometry, namely when there is only one convex set $I_2$ ahead. Also note that the set $\mathrm{Int} \{  \cap_{\pi \in \widetilde{\mathbb{P}_1 }(p)} \overline{^1H^-_\pi} \} $ is the past set of $p$ for the non-relative backward Funk geometry, namely when there is only one convex set $I_1$ ahead.
 
\begin{proposition} \label{prop:eq}
We have the equivalences: 
\[
p<q \iff \widetilde{\mathbb{P}_2}^c(p) \cup \widetilde{\mathbb{P}_1}(p)  \subsetneq \widetilde{\mathbb{P}_2}^c(q) \cup \widetilde{\mathbb{P}_1}(q) \iff  \widetilde{\mathbb{P}_2}(p) \cup \widetilde{\mathbb{P}_1}^c(p)  \supsetneq \widetilde{\mathbb{P}_2}(q) \cup \widetilde{\mathbb{P}_1}^c(q).
\]

\end{proposition}

\begin{proof} 
Suppose $p < q$.  We claim that every $\pi \in \widetilde{\mathbb{P}_2}^c(p) \cup \widetilde{\mathbb{P}_1}(p)
$ is an element of $\widetilde{\mathbb{P}_2}^c(q) \cup \widetilde{\mathbb{P}_1}(q) $. This follows from the two inclusions $\widetilde{\mathbb{P}_2}^c(p) \subsetneq \widetilde{\mathbb{P}_2}^c(q)$ and 
$\widetilde{\mathbb{P}_1}(p) \subsetneq \widetilde{\mathbb{P}_1}(q)$.
The first inclusion follows from the fact that the ray from $p$ through $q$ hits the open convex set $\widetilde{I_2}$, and the second follows from the fact that the ray from $q$ through $p$ hits $\widetilde{I_1}$.

 To see the strict inclusion when $p < q$, choose a hyperplane  in $\widetilde{\mathbb{P}_2}^c(q) \cup \widetilde{\mathbb{P}_1}(q)
$ intersecting $]p,q[$.  Such a hyperplane is not in $\widetilde{\mathbb{P}_2}^c(p) \cup \widetilde{\mathbb{P}_1}(p)$.  

Next suppose $ \widetilde{\mathbb{P}_2}^c(p) \cup \widetilde{\mathbb{P}_1}(p) \subsetneq \widetilde{\mathbb{P}_2}^c(q) \cup \widetilde{\mathbb{P}_1}(q)
$.  Then the following inclusion 
\[
{\mathfrak I}_2^+(p) \supsetneq {\mathfrak I}_2^+ (q)
\]
follows from the characterization (\ref{char:I}) of 
${\mathfrak I}_2^-(x)$.

Hence $p$ is in the past of $q$, and thus $p < q$.

The second inclusion is simply the inclusion induced by taking the complements of the first inclusion.
\end{proof}

  We introduce the notation $$\mathbb{P}_{12}(p)=  \widetilde{\mathbb{P}_2}(p) \cup \widetilde{\mathbb{P}_1}^c(p).$$ 
Thus the statement of the proposition above becomes
\[
p<q \iff \mathbb{P}_{12}(p) \supsetneq \mathbb{P}_{12}(q).
\]

  \medskip
  $\widetilde{\mathcal{P}_2}(p)$ is the set of supporting hyperplanes to $\widetilde{K_2}$ so that
 \[
 \widetilde{\mathcal{P}_2}(p) =  \mathbb{P}_{12}(p) \cap \widetilde{\mathcal{P}_2} 
 \]

$\widetilde{\mathcal{P}_2}(p)$ is  the set of supporting hyperplanes to $K_2$ at the points of $K_1^2(p)$, the future set of $p$ in $K_1^2$.
We note that a supporting hyperplane $\pi$ to $\widetilde{I_2}$ that contains $p$ does not belong to $\widetilde{\mathcal{P}_2}(p)$.

\begin{corollary}\label{p:inclusion1}
For any two points $p$ and $q$ in $\Omega$, we have 
\[
p< q \Rightarrow \widetilde{\mathcal{P}_2}(p) \supset\widetilde{\mathcal{P}_2}(q).
\]
\end{corollary}

\begin{proof}
This follows  from the fact that $\widetilde{\mathcal{P}_2}(p) = (\mathbb{P}_{12}(p) \cap \widetilde{\mathcal{P}_2}) \supset (\mathbb{P}_{12}(q) \cap \widetilde{\mathcal{P}_2}) =  \widetilde{\mathcal{P}_2}(q)$.  
\end{proof}

In Corollary \ref{p:inclusion1}, the strict inclusion cannot be expected, as can be seen from the following example in ${\mathbb R}^2$ where we have $\widetilde{\mathcal{P}_2}(p) = \widetilde{\mathcal{P}_2}(q)$:

$K_1$ is the line with equation $\{y=-3\}$, bounding the convex half-space $\{y<-3\}$; 

$K_2$ is the convex curve in $\mathbb{R}^2$ which is the union of the rays $\{y=x, \ y>0\}$ and $\{y=-x, \ y>0\}$, $p=(0,-2)$ and $q=(0,-1)$.

\begin{corollary}\label{p:inclusion112}
 Let $p,q,r$ be three points in $\Omega$. If $p < q$ and $q < r$, then $p < r$.
\end{corollary}

\begin{proof}  This follows from  Proposition \ref{prop:eq} which gives:
\[
p< q \mbox{ and } q < r \Leftrightarrow \mathbb{P}_{12}(p) \supsetneq  \mathbb{P}_{12}(q) \supsetneq  \mathbb{P}_{12}(r).
\] 

\end{proof}

Now we can define the timelike relative Funk distance $F_1^2(p,q)$ on the subset $\Omega_{\leq}$ of $\Omega\times \Omega$.

\begin{definition}[The timelike relative Funk distance]\label{def:Funk1}
The timelike relative Funk distance  $F_1^2(p, q)$ is first defined on pairs of  
distinct points $p, q$ in $\Omega$ satisfying $p < q$ by the formula
\[
F_1^2(p, q) = \log \frac{d(p, b(p,q))}{d(q, b(p,q))}
\]
  where $b(p,q)$ is the first intersection point of the ray $R(p,q)$ with $K_2$. As before,
$d(\cdot \, ,\cdot)$ denotes the Euclidean distance.

Note that the value of $F_1^2(p, q)$ is strictly positive for any pair $p,q$ satisfying $p<q$. 

We extend the definition of $F_1^2(p, q)$ to the case where $p=q$, setting in this case $F_1^2(p, q) = 0$.
\end{definition}

Using the convexity of $\widetilde{K_2}$, we now give a variational characterization of the quantity $F_1^2(p, q)$. 

 Let $p$ and $q$ be two points in $\Omega$ such that $p<q$.  Let $\pi_0$ be a supporting hyperplane to $K_2$ at $b(p,q)$.  For $x$ in $\mathbb{R}^n$, let $\Pi_{\pi_0}(x)$ be the foot of the Euclidean perpendicular from the point $x$ onto that hyperplane. 
 
From the similarity of the two Euclidean triangles  $\triangle (p, \Pi_{\pi_0}(p), b(p,q))$
and $\triangle (q, \Pi_{\pi_0}(q), b(p,q))$, we have
\[
\log \frac{d(p, b(p,q))}{d(q,  b(p,q))} = \log \frac{d(p, \pi_0)}{d(q,  \pi_0)}.
\]

For any unit vector $\xi$ in $\mathbb{R}^n$ and for any $\pi \in {\mathcal P}(p)$, we set
\[T(p,\xi, \pi)=\pi \cap \{p+t \xi  \,\, | \,\, t > 0\}\] 
if this intersection is non-empty.

For $p <q$ in $\mathbb{R}^n$, consider the vector $\xi = \xi_{pq}= \frac{q-p}{\Vert q-p\Vert}$ where the norm is, as before, the Euclidean one.  

 We then have $T(p,\xi_{pq}, \pi_{b(p,q)}) = b(p, q)\in R(p, q) \cap K_2$.
 
 In the case where $\pi\in\widetilde{\mathcal{P}_2}(q)$ is not a supporting hyperplane of $\widetilde{K_2}$ at $b(p,q)$, 
 the point $T(p, \xi_{pq}, \pi)$ lies outside $\overline{\widetilde{K_2}}$ and,  again by the similarity of the Euclidean  triangles
$\triangle (p, \Pi_{\pi}(p), T(p, \xi_{pq}, \pi))$ and $\triangle (q, \Pi_{\pi}(q), T(p,\xi_{pq}, \pi))$, we get
\[
\frac{d(p, \pi)}{d(q, \pi)} =\frac{d(p, T(p,\xi_{pq},\pi))}{d(q, T(p,\xi_{pq}, \pi))}.
\]

As $\pi$ varies in $\widetilde{\mathcal{P}_2}(q)$, the farthest point from $p$ on the ray $R(p,q )$ of the form $T(p,\xi_{pq}, \pi)$ is $b(p, q)$, and this occurs when $\pi$ supports $\widetilde{K_2}$ at $b(p, q)$.
This in turn says that  a hyperplane $\pi_{b(p, q)}$ which supports $\widetilde{K_2}$ at $b(p, q)$ minimizes the ratio 
\[
\frac{d(p, T(p,\xi_{pq}, \pi))}{d(q, T(p,\xi_{pq}, \pi))} 
\]
among all the elements of $\widetilde{\mathcal{P}_2}(q)$
and thus we obtain  
\begin{proposition}\label{prop:star1} For all $p<q$, we have
 \[\log F_1^2(p,q) = \inf_{\pi \in \widetilde{\mathcal{P}_2}(q)} \log \frac{d(p, \pi)}{d(q,   \pi)}.
 \] 
\end{proposition}

Note that the statement is much similar to that of Proposition \ref{prop:star}.   The similarity illustrates that the relative timelike Funk geometry is a \textit{restriction} of the timelike Funk geometry.

The fact that the function $F_1^2(p,q)$ satisfies the  time inequality follows from an argument similar to the one used in proving  Proposition 
\ref{time-ineq}, in the light of Corollary \ref{p:inclusion1}.

\begin{proposition}[Time inequality]
For any three points $p, q$ and $r$ in $\Omega$, satisfying $p< q< r$, we have
\[
F_1^2(p, q) + F_1^2(q, r) \leq F_1^2(p, r).
\]
\end{proposition}
 
The following proposition is an analogue of Proposition \ref{prop:geo1} that concerns (non-relative) timelike Funk geometries, and it is proved in the same way:

\begin{proposition}[Geodesics]\label{prop:geo2} A timelike relative Funk geometry $F_1^2$ defined on a set $\Omega_\leq $ associated with two disjoint convex hypersurfaces $K_1$ and $K_2$ in $\mathbb{R}^n$ satisfies the following:
\begin{enumerate}
\item \label{eucl2} The Euclidean segments  in $\Omega$ that are of the form $[p,q]$ where $p<q$ are $F_1^2$-geodesics. 
\item Any Euclidean segment $[p,b)$ from a point $p$ in $\Omega$ to a point $b$ in $\partial K$, equipped with the metric induced from the timelike distance $F_1^2$, is isometric to a Euclidean ray. 
\item The Euclidean segments in (\ref{eucl2}) are the unique $F_1^2$-geodesic segments
if and only if the pair $(K_1,K_2)$ satisfies the following property: there is no nonempty open Euclidean segment  contained in the subset $K_1^2$ of points in $K_2$ facing $K_1$.
\end{enumerate}
\end{proposition} 

\section{The timelike Finsler structure of the timelike Euclidean relative Funk distance}\label{s:t-Funk-Finsler}

 In this section, as in  \S \ref{s:ETRF}, $\Omega$ is the space underlying the timelike Funk geometry associated with two disjoint convex hypersurface $K_1$ and $K_2$ in $\mathbb{R}^n$. We show that the timelike Euclidean relative Funk distance associated with $K_1$ and $K_2$ is timelike Finsler.

With every point $p$ in $\Omega$, we shall associate a  \emph{timelike Minkowski functional} $f_{F_{1}^{2}}(p,v)$ defined on the subset of the tangent space $T_p\Omega$ of $\Omega$ at $p$ consisting of the non-zero vectors $v$ satisfying 
\begin{equation}\label{eq:v}
 p + tv \in \mathfrak{I}_2^+(p) \mbox{ for some } t > 0
\end{equation}
where $\mathfrak{I}_2^+(p)$ is as before the future of $p$.

 We denote by  $C_2^+(p) \subset T_p \Omega$ the set of  vectors $v$ that satisfy Property (\ref{eq:v}) or are the zero vector.  We define the function $f_{F_{1}^{2}}(p,v)$ for  $p \in \Omega$ and  for any nonzero vector $v \in C_2^+(p)$  by the following formula:
\begin{equation}\label{eq:P21TRF}
f_{F_{1}^{2}}(p, v) = \inf_{\pi \in \widetilde{\mathcal{P}_2}(p)} \frac{\langle v, \eta_\pi \rangle }{d(p, \pi)},
\end{equation}
 where for each plane $\pi$ in $\widetilde{\mathcal{P}_2}(p)$, $\eta_\pi$ is the unit tangent vector at $p$ perpendicular to $\pi$ and pointing toward $\pi$. We define $f_{F_{1}^{2}}(p,0)=0$ if $v =0$. We shall show that this defines a timelike Minkowski functional on the tangent space of $\Omega$ at $p$ and that this functional is associated with a timelike Finsler geometry underlying the relative timelike Euclidean Funk distance $F_1^2$.

In the same way as for the Finsler structure of the timelike Euclidean (non-relative) Funk geometry (see Equations (\ref{eq:P1}) and (\ref{eq:P211})), we have, for every nonzero vector $v \in C^+(p)$:
\begin{equation}\label{eq:P333}
f_{F_{1}^{2}}(p, v) = \frac{\|v\|}{\inf \Big[ t \,\, | \,\, p + t \frac{v}{\|v\|} \in  \overline{\widetilde{K_2}} \Big]}=\sup\{\tau:p+v/\tau \in \overline{\widetilde{K_2}}\}.
\end{equation}

The following can be easily checked.

\begin{proposition}
The functional $f_{F_{1}^{2}}(p, v)$ defined on the open cone $C_2^+(p)$ in $T_p \Omega \cong {\mathbb R}^n$ satisfies all the properties required in Definition \ref{def:Min}  by a timelike Minkowski functional.  
\end{proposition}

Now we repeat the argument in \S \ref{s:Funk-Finsler}, to set up a timelike space using the Finsler structure. We say that a piecewise  $C^1$ curve  $\sigma: [0,1]\to  \mathbb{R}^n$, $t\mapsto \sigma(t)$ is {\it timelike} if at each 
time $t\in [0,1]$ the tangent vector $\sigma'(t)$ is an element of the cone $C_2^+(\sigma(t)) \subset T_{\sigma(t)} \mathbb{R}^n$. We shall follow the same scheme as in \S \ref{s:Timelike-Finsler} to show that the timelike Euclidean relative Funk distance is Finsler.

\begin{definition}[The partial order relation]
Suppose that $p$ and $q$ are two points in $\mathbb{R}^n$. We write $p \prec q$, and we say that \emph{$q$ is in the $\prec$-future of $p$}, if there exists a timelike piecewise $C^1$ curve  $\sigma: [0,1] \to \mathbb{R}^n$ joining $p$ to $q$.
\end{definition}

\begin{proposition}
The two order relations $<$ and $\prec$ coincide, namely, for any two points $p$ and $q$ in $\Omega$, we have 
\[
p < q \Leftrightarrow p \prec q.
\]
\end{proposition} 
The proof is the same as that of Proposition \ref{order-equiv} except  that $\mathfrak{I}^+(p)$ needs to be replaced by $\mathfrak{I}_2^+(p)$.

Similarly to what we did in \S \ref{s:Funk-Finsler}, we denote by $\delta_1^2$ the timelike intrinsic distance function associated with this timelike Finsler structure:
\begin{equation}
\delta_1^2(p, q) = \sup_{\sigma} {\rm Length}(\sigma)
\end{equation}
where the supremum is taken over all the timelike piecewise $C^1$ curves $\sigma:[0,1]\to \Omega$ satisfying $\sigma(0)=p$ and $\sigma(1)=q$.  By a proof similar to that of Lemma \ref{lemma:finite}, the intrinsic distance $\delta(p, q)$ for $p<q$ is finite. 

Thus, the domain of definition $\Omega_<$ defined with the partial order $<$ for  the timelike Funk distance $F_1^2$  and the domain of definition $\Omega_\prec$ for the timelike distance function $\delta_1^2$  coincide. We shall prove the equality $\delta_1^2(p, q) =F_1^2 (p, q)$ for any pair $p < q$ in $\Omega_{\leq}$. We state this as follows:

\begin{theorem}\label{th:distance11}
The value of the timelike distance $\delta_1^2(p, q)$ for a pair $(p,q)\in \Omega_{\leq}$ coincides with $F_1^2(p, q)$. That is, we have
\[
F_1^2(p, q) = \delta_1^2(p, q).
\]
\end{theorem}

In other words, we have the following
\begin{theorem}\label{Funk-Finsler-2} 
The relative timelike Funk geometry is a timelike Finsler structure with associated Minkowski functional $f_{F_{1}^{2}}(p,v)$.
\end{theorem}

Note that with this identification $F_1^2=\delta_1^2$, given a pair of points $p, q$ satisfying $p < q$, there always exists a distance-realizing (length-maximizing) geodesic from $p$ to $q$, since the Euclidean segment $[p,q]$ is an $F_1^2$-geodesic.

\begin{proof}[Proof of Theorem \ref{th:distance11}] The proof is similar to the one of Theorem \ref{th:distance}:
Given a pair $(p, q)$ with $p<q$, consider the geodesic ray $R(p, q)$ from $p$ through $q$ and let $b(p, q) \in K_1^2$ be the first intersection point of this ray with the convex set  $\overline{\widetilde{K_2}}$. Parameterize proportionally to arc-length the Euclidean segment $[p,q]$ by a path $\sigma(t)$ with parameter $t$ in $[0,1]$, with $\sigma(0)=p, \sigma(1)=q.$  Then we have 
\[
\int_0^1 f_{F_1^2}(\sigma(t), \sigma'(t)) \, dt = \log \frac{d(p, b(p,q))}{d(q, b(p,q))} = F_1^2(p, q),
\] 
since 
\[
\frac{d}{dt} \log \frac{d(p, b(p,q))}{d(\sigma(t), b(p,q))} = f_{F_1^2}(\sigma(t), \sigma'(t)).
\]
Taking the supremum over the set of paths from $p$ to $q$, we obtain the inequality
\[
\delta_1^2(p, q) \geq  F_1^2(p, q).
\]

We need to show a monotonicity lemma similar to Lemma \ref{lem:another1} for the intrinsic distance $\delta_1^2$.

Let 
$\widehat{I}_2 \supset I_2$ be an open convex set containing $I_2$, let  $\widehat{K}_2$ be its bounding hypersurface,  $\widehat{P}_1^2$ the timelike Minkowski functional associated with the pair $(I_1, \widehat{I}_2)$ and $\widehat{\delta}_1^2$ the associated intrinsic distance. (Note   that the domains of definition of   $\widehat{P}_1^2$ and $\widehat{\delta}_1^2$ contain those of $f_{F_{1}^{2}}$ and $\delta_1^2$ respectively.) 
\begin{lemma}\label{lem:another2}   For $(p,q)$ in the domains of definition of both intrinsic distances $\delta_1^2$ and $\widehat{\delta}_1^2$, we have
\[
\widehat{\delta}_1^2(p, q) \geq \delta_1^2(p, q).
\]
\end{lemma}

The proof is, with an adaptation of the notation, the same as that of Lemma \ref{lem:another1}.

\noindent {\it Proof of Theorem  \ref{th:distance11} continued}.--- For $(p, q) \in \Omega_<$,  let 
\[\widehat{I}_2 = H^+_{\pi_{b(p, q)}},\]
 where $H^+_{\pi_{b(p, q)}}$  is 
the open half-space  bounded by a hyperplane $\pi_{b(p, q)}$ supporting $K_1^2$ at $b(p, q)$ and containing $\widetilde{I}_2$. The open set $\widehat{I}$ is equipped with its intrinsic distance $\widehat{\delta}$. We now apply Lemma \ref{lem:another1} to this setting where a convex set  $\widehat{I}_2$  
contains $I_2$, and obtain $\widehat{\delta}_1^2 \geq \delta_1^2$.

For the open half-space $\widehat{I}_2 = H^+_{\pi_{b(p, q)}}$, the values of $F_1^2(p, q)$, $\widehat{F}_1^2(p, q)$ and $\widehat{\delta}_1^2(p, q)$ all coincide. Indeed, under the hypothesis $\widehat{I}_2 = H^+_{\pi_{b(p, q)}}$, the set $
\widetilde{\mathcal{P}_2}$ of supporting hyperplanes consists of the single element $\pi_{b(p, q)}$, and the Euclidean segment $\sigma$ from $p$ to $q$ is length-maximizing, since every timelike path for $\widehat{I}_2$ is $\widehat{F}_1^2$-geodesic. 
 
By combining the above observations, we have 
\[
 F_1^2(p, q) = \widehat{F}_1^2(p, q) = \widehat{\delta}_1^2(p, q) \geq \delta_1^2(p, q) \geq F_1^2(p, q)
\]
and  the equality $\delta_1^2(p, q) =F_1^2(p, q)$ follows.
\end{proof}

\section{The timelike Euclidean relative reverse  Funk geometry and its Finsler structure}\label{s:TRRF}
We continue using the notation of \S \ref{s:ETRF} and \S \ref{s:t-Funk-Finsler} associated with two convex subsets $K_1$ and $K_2$ of $\mathbb{R}^n$.

\begin{definition} \label{def:123} The timelike Euclidean relative reverse Funk geometry is the function $\overline{F_{1}^{2}}$ defined for $p$ and $q$ in $\Omega$ satisfying $p<q$  by
\begin{equation}
 \overline{F_{1}^{2}}(p,q)=F_2^1(q,p)
\end{equation}
where $F_2^1(q,p)$ is the timelike Euclidean relative Funk metric associated with the pair $(K_2, K_1)$, that is, here, the convex set $K_1$ represents the future and the convex set $K_2$ represents the past, and where $p$ lies in the future of $q$ relatively to this ordered pair. (In particular, the domain of definition of  $\overline{F_{1}^{2}}$ is equal to the domain of definition of $F_1^2$.) 
\end{definition}

With the notation introduced in \S \ref{s:t-Funk-Finsler}, we have:   
\begin{equation} \label{future-past-symmetry} 
q \in \mathfrak{I}_2^+(p) \Leftrightarrow p \in \mathfrak{I}_1^+(q).
\end{equation}

For every point $p$ in $\Omega$, we have a \emph{timelike Minkowski functional} $f_{F_{2}^{1}}(p,v)$ defined on the subset of the tangent space $T_p\Omega$ of $\Omega$ at $p$ consisting of the non-zero vectors $v$ satisfying 
\[
 p + tv \in \mathfrak{I}_1^+(p) \mbox{ for some } t > 0.
\]
 We denote by  $C_1^+(p) \subset T_p \Omega$ the union of tangent vectors $v$ that satisfy this property or are the zero vector.  From the definition, there is a symmetry between $C_1(p)$ and $C_2(p)$ in the sense that
 \[
 v \in C_1(p) \Leftrightarrow -v \in C_2(p).
 \]
This follows from the equivalence (\ref{future-past-symmetry}) remarked above.
 
 We define the function $f_{F_{2}^{1}}(p,v)$ for  $p \in \Omega$ and for any nonzero  $v \in C_1^+(p)$  by the following formula:
\begin{equation}\label{eq:P21TRRF}
f_{F_{2}^{1}}(p, v) = \inf_{\pi \in \widetilde{\mathcal{P}_1}(p)} \frac{\langle v, \eta_\pi \rangle }{d(p, \pi)}
\end{equation}
 where for each hyperplane $\pi$ in $\widetilde{\mathcal{P}_1}(p)$, $\eta_\pi$ is the unit tangent vector at $p$ perpendicular to $\pi$ (with respect to the underlying Euclidean metric) and pointing toward $\pi$. 
 
 We extend the definition by setting $f_{F_{1}^{2}}(p,0)=0$ when $v =0$. 
 
 In the same way as for the geometries that were considered previously, this defines a timelike Minkowski functional, and this functional is associated with a timelike Finsler geometry underlying the timelike Funk distance $F_2^1$. 

We shall use the following definition in \S \ref{def:H}:
\begin{definition} The timelike Minkowski functional $f_{\overline{F_{2}^{1}}}(p,v)$ for the timelike Euclidean relative reverse Funk geometry $\overline{F_{1}^{2}}$ is the function
\[
f_{\overline{F_{2}^{1}}}(p,v) = f_{F_{2}^{1}}(p, -v) .
\]
for $v \in C_2^+(p) (= - C_1^+(p)) $.
\end{definition}

\section{The timelike Euclidean Hilbert geometry}\label{def:H}

We continue using the notation introduced in \S \ref{s:ETRF}:  $I_1$ and $I_2$ are
two disjoint open (possibly unbounded) convex sets  in ${\mathbb R}^n$ bounded by disjoint convex hypersurfaces $K_1$ and $K_2$ respectively and 
 $\overline{K_1}= K_1\cup I_1$ and  $\overline{K_2}= K_2\cup I_2$.
 The latter are the closures of $I_1$ and $I_2$ respectively.

We shall define the timelike Euclidean Hilbert geometry $H(p,q)$  associated with the ordered pair $K_1,K_2$. Its underlying space $\Omega$ is the same as the one of  the timelike Euclidean relatively Funk geometry defined in \S \ref{s:ETRF}, that is, $\Omega$ is the union in $\mathbb{R}^n$ of the open segments of the form $]a_1,a_2[$  such that $a_1\in K_1$, $a_2\in K_2$ satisfying $]a_1,a_2[\cap (K_1\cup K_2)=\emptyset$ for $i=1,2$ and for which there is no supporting hyperplane $\pi$ to $K_1$ or to $K_2$ with $]a_1,a_2[\subset \pi$.
   
 Referring to \S \ref{s:ETRF}, we shall use the two timelike relative Funk metrics, $F_1^2$ and $F_2^1$, both defined on $\Omega$,  but we shall always consider $K_1$ as representing the past and $K_2$ the future, except if the contrary is explicitly specified. 
 
 In particular, the \emph{partial order relation} on $\Omega$ that underlies the timelike Hilbert geometry $H(p,q)$ is the same as the one associated with the relative Euclidean Funk metric  with respect to the pair $K_1,K_2$ as an ordered pair. The relative future and relative past of a point $p$ in $\Omega$ are defined accordingly, as in Definitions \ref{def:o} and \ref{def:rp}.

\begin{definition}[Timelike Euclidean Hilbert geometry]\label{def:Hilbert}
The timelike Euclidean Hilbert distance is defined on ordered pairs $(p,q)\in \Omega$ satisfying $p<q$ by
 \[H(p,q)=\frac{1}{2}(F_1^2(p,q)+\overline{F_1^2}(p,q)).\]

 The definition is extended to the case where $p=q$ by setting $H(p,q)=0$.
 \end{definition}
   
   Even though this definition of $H$ depends on the \emph{ordered 
   pair} $K_1, K_2$, it is clear that its values are independent of the order. For that reason, we choose the notation $H$.

The fact that the timelike Hilbert geometry satisfies the time inequality follows from the definition of the timelike Hilbert geometry as a sum of two timelike relative Funk geometries that both satisfy the time inequality.
%
%

The timelike Hilbert geometry satisfies some properties which follow from those of a timelike Funk geometry. In particular, we have the following:

\begin{proposition}\label{Hibert-geo}
(a) In a timelike Hilbert geometry $H$ associated with an ordered pair of convex hypersurfaces $K_1,K_2$, the Euclidean segments of the form $]a_1,a_2[$ such that
\begin{enumerate}
\item $a_1 \in K_1$ and $a_2 \in K_2$; 
\item $]a_1,a_2[$  is not contained in any supporting hyperplane to $K_1$ or to $K_2$;
\item the open segment $]a_1,a_2[$ is in the complement $\Omega$ of $\overline{K_1}\cup \overline{K_2}$ 
\end{enumerate}
 are $H$-geodesics. Furthermore, each such geodesic is isometric to the real line.
(We recall that, as it is always the case in timelike spaces, it is understood that the segments $]a_1,a_2[$ are oriented from $a_1$ to $a_2$. Traversed in the reverse sense, they are not geodesics.)

 (b) The following two properties are equivalent:
 \begin{enumerate}
 \item the oriented Euclidean segments contained in the segments  of the form $]a_1,a_2[$  satisfying the above three properties are the unique $H$-geodesics;
 \item 
there are no segments $]a_1,a_2[$ satisfying the above three properties, with $a_1$ in the interior of an open nonempty Euclidean segment $J_1\subset K_1$ and $a_2$ in the interior of an open nonempty segment $J_2\subset K_2$, with $J_1$ and $J_2$ coplanar.
\end{enumerate}
\end{proposition}
The proof is an adaptation of that of the non-timelike Hilbert metric (cf. \cite{G} or \cite{PT}), and we omit it.
 
We now express the timelike Hilbert distance using the cross ratio, as in the usual (non-timelike) Hilbert geometry.

Recall that if $a,b,c,d$ are four distinct points lying in that order on a Euclidean line, their cross ratio $[a,b,c,d]$ is defined by
\begin{equation}\label{eq:cross} [a,b,c,d]=\frac{\vert b-d\vert}{\vert c-d\vert}\frac{\vert c-a\vert}{\vert b-a\vert}.
\end{equation}

The following proposition follows easily  from the definition of the cross ratio and the timelike Euclidean Hilbert distance:

\begin{proposition}\label{prop:cross1}
For any two points $p$ and $q$ in $\Omega$ satisfying $p<q$,  their timelike Euclidean Hilbert distance is also given by
\[H(p,q)= \frac{1}{2}\log [a_1,p,q,a_2]\]
where $a_1$ and $a_2$ satisfy  $[a_1,a_2]\cap K_i=a_i$ for $i=1,2$.
\end{proposition}
 
 With this form of the definition of the timelike Euclidean Hilbert geometry, we see that the projective transformations of $\mathbb{R}^n$ that preserve (setwise) each of the two convex sets  $I_1$ and $I_2$ are isometries for the timelike Hilbert distance. 
 
We point out two 2-dimensional examples of timelike Hilbert geometries. Higher-dimensional analogues also hold.
\begin{example}[The strip]\label{ex:strip}
Let $\Omega$ be a region contained by two parallel lines in the plane ${\mathbb R}^2 = \{(x, y)\}$, namely, $\Omega$ is the complement of two 
half-spaces $H_1 = \{ y \leq -1 \}, H_2 = \{ y \geq 1\}$. Then any timelike curve is a geodesic for the timelike Hilbert geometry. In this setting,  a curve is timelike if at each point the tangent  vectors are not horizontal. 

Consider the nearest point projection $\Pi: \Omega \rightarrow (-1, 1)$ onto the interval (-1,1) of the $y$-axis.  Then the Hilbert distance $H(p, q)$ for $p< q$ is equal to $H_{(-1, 1)}(\Pi(p), \Pi(q))$ where 
\[
H_{(-1, 1)}  (a, b)= \frac12  \log \frac{a -1}{b-1}  \frac{b+1}{a+1}
\]
 is the Hilbert distance for the interval.  This metric is sometimes called the ``one-dimensional hyperbolic metric" as this is the Klein-Beltrami model of the hyperbolic space ${\mathbb H}^1$.  Notice that $\Omega$ is concave as well as convex in ${\mathbb R}^2$.
\end{example}

\begin{example}[The half-space] 
The half-space corresponds to the limiting case of the strip discussed in Example \ref{ex:strip} above, $\Omega =  {\mathbb R}\times (-a, 1) $, as $a \rightarrow \infty$.  Then the Hilbert timelike distance 
\[
H_{(-a, 1)}(p, q) = \frac12 \log \frac{\Pi(p) -1}{\Pi(q)-1}  \frac{\Pi(p)+a}{\Pi(q)+a}
\]
converges to (half of) the timelike Funk distance
\[
F(p, q)= \log  \frac{\Pi(p) -1}{\Pi(q)-1} 
\]
which is the timelike Funk distance for the half-space ${\mathbb R}^2 \setminus \{ y \geq 1 \}$. We will come back to this example later.
\end{example}

\begin{remark}  Our approach to the timelike Euclidean Hilbert geometry, based on the relative Euclidean Funk geometry, is different from that of Busemann in \cite{B-Timelike}. In fact, 
 Busemann, in \S 8 of his paper \cite{B-Timelike}, works in the projective space, and the geometry which he obtains is a \emph{local} timelike geometry (the order relation is only locally defined). Thus, the Hilbert geometry he obtains is \emph{locally} timelike.
 
 One important result that Busemann obtains (his Theorem (3) p. 47) is that in the case where the convex sets $K_1$ and $K_2$ are strictly convex, the isometry group of a locally timelike Hilbert geometry is the group of restrictions of projective transformations of the ambient projective space that preserve the given convex set.  
 
Busemann then defines a timelike Funk geometry 
 associated with a convex hypersurface $K$ contained in an affine space ${\mathbb A}^n$  using his locally timelike
 Hilbert geometry, namely, it becomes the geometry  associated with a pair $K_1,K_2$ where $K_1$ is the hyperplane at infinity ${\mathbb R} {\rm P}^{n-1} $
 in the projective space ${\mathbb R} {\rm P}^n={\mathbb A}^n \cup {\mathbb R} {\rm P}^{n-1}$.  The set $K_1$ is the collection of points which are ``infinite distance away'' from any 
 pair of points in ${\mathbb A}^n \setminus K_1$, in the sense that for any pair of points $p,q$ with $p < q$ (the order relation when $K_2$ is the future set), we have $\frac{d(p, a_1)}{d(q, a_2)} = 1$. In that case, and using the notation of Definition \ref{def:Hilbert}, the Hilbert distance from $p$ to $q$ associated with the pair $K_1,K_2$ is just  the Funk distance from $p$ to $q$ associated with the convex set $K_2$, up to a constant.
 \end{remark}
  
 \section{The timelike Finsler structure of the timelike Hilbert geometry}\label{s:TFSTHG}
In this section, we show that the  timelike Hilbert distance $H(p, q)$ introduced in \S \ref{def:H} is a timelike Finsler metric, and we give its timelike Minkowski functional. 

We continue using the notation introduced in \S \ref{s:t-Funk-Finsler} for the Finsler structure of the timelike Euclidean relative Funk distance.
 
Consider a point $p$ in $\Omega$ so that the associated cone $C_2^+(p)\subset T_p(\Omega)$ (which, we recall, is equal to the cone  $- C_1^+(p)$) is nonempty.  We denote by $C(p)$ the set $C_2(p) = -C_1(p) \subset T_p \Omega$.  Following the notation of \S\,\ref{s:Funk-Finsler} that concerns the infinitesimal Finsler metric associated with a timelike Funk geometry, we define a linear functional  on $C(p)$ by the formula: 
\begin{equation} \label{eq:Hilbert2}
f_H(p, v) = f_{F_{1}^{2}}(p,v) + f_{F_{2}^{1}}(p,-v),\end{equation}
or, equivalently,
\begin{equation} \label{eq:Hilbert1}
f_H(p, v) = f_{F_{1}^{2}}(p,v) + f_{\overline{F_{1}^{2}}}(p,v).\end{equation}
where $f_{F_{1}^{2}}$ and $f_{\overline{F_{1}^{2}}}$ are the timelike Minkowski norms on the tangent spaces associated with the timelike relative Funk geometry and the timelike reverse Funk geometry defined by $K_1$ and $K_2$.

 Now we follow the outline used in \S \ref{s:Funk-Finsler}, to set up a timelike space using the Finsler structure $f_H$. We say that a piecewise  $C^1$ curve  $\sigma: [0,1]\to  \mathbb{R}^n$, $t\mapsto \sigma(t)$ is {\it timelike} if at each 
time $t\in [0,1]$ the tangent vector $\sigma'(t)$ is an element of the cone $C_2^+(\sigma(t)) \subset T_{\sigma(t)} \mathbb{R}^n$.

\begin{definition}[The partial order relation]
Suppose that $p$ and $q$ are two points in $\Omega$. We write $p \prec q$, and we say that \emph{$q$ is in the $\prec$-future of $p$}, if there exists a timelike piecewise $C^1$ curve  $\sigma: [0,1] \to \mathbb{R}^n$ joining $p$ to $q$.
\end{definition}
 
As in the situation studied in \S \ref{s:Funk-Finsler}, the following holds in the present setting as well, and the proof is the same as that of Proposition \ref{order-equiv}, with $\mathfrak{I}^+(p)$ replaced by $\mathfrak{I}_2^+(p)$. 

\begin{proposition}
The two order relations $<$ and $\prec$ coincide; namely, for any two points $p$ and $q$ in $M$, we have 
\[
p < q \Leftrightarrow p \prec q.
\]
\end{proposition} 

As in \S \ref{s:Timelike-Finsler}, we denote by $\delta$ the timelike intrinsic distance function associated with this timelike Finsler structure:
\begin{equation}
\delta(p, q) = \sup_{\sigma} {\rm Length}(\sigma)
\end{equation}
where the supremum is taken over all the timelike piecewise $C^1$ curves $\sigma:[0,1]\to \Omega$ satisfying $\sigma(0)=p$ and $\sigma(1)=q$.  As in Lemma \ref{lemma:finite}, we prove that for all $p<q$, we have $\delta(p, q)<\infty$. This implies that the domain of definition $\Omega_<$ associated with the partial order $<$ for  the timelike Hilbert distance $H$  and the domain of definition $\Omega_\prec$ for the  timelike distance function $\delta$  coincide. 

Now we prove the equality $\delta(p, q) =F(p, q)$ for any pair $(p,q)$ satisfying $p < q$ in $\Omega_< = \Omega_\prec$. We state this as follows:

\begin{theorem}\label{Hilbert-Finsler} 
The timelike Hilbert geometry has an underlying timelike Finsler structure given by the 
Minkowski functional $f_H$ defined in (\ref{eq:Hilbert1}).
\end{theorem}

\begin{proof} 
Let $(p, q)$ be an element in $\Omega_<$. In what follows, when we talk about a Euclidean segment $[p,q]$ joining $p$ to $q$, it is understood that this segment is oriented from $p$ to $q$. We parametrize such a segment $[p,q]$ by $x(t)$, $0\leq t\leq 1$ and the same segment traversed in the opposite direction, $[q,p]$, by $y(t)=x(1-t)$.

Recall that the Euclidean segment $[p,q]$ is an $F_1^2$-geodesic, and the Euclidean segment $[q,p]$, is an $F_2^1$-geodesic.
Thus, we have
\[F_1^2(p,q)=\int_{[p,q]} f_{F_{1}^{2}}(x,x')dx\]
and 
\[F_2^1(q,p)=\int_{[q,p]} f_{F_{2}^{1}}(x,x')dx.\]

Since the segment $[q,p]$ is the interval $[p,q]$ traversed in the opposite direction, we have
\[
\int_{[q, p]} f_{F_{2}^{1}}(y,y')dy=\int_{[p,q]} f_{\overline{F_{1}^{2}}}(x,x')dx.\]
 Thus, we obtain
\begin{equation} \label{eq:H1}
H(p,q)=\int_{[p,q]} \left(f_{F_{1}^{2}}(x,x')+ f_{\overline{F_{1}^{2}}}(x,x')\right)dx\leq \delta(p,q).
\end{equation}
Furthermore, if $\gamma$ is now an arbitrary path in the domain of definition of $H$ joining $p$ to $q$, we have 
\begin{equation}\label{eq:H21}
\int_{\gamma}f_{F_{1}^{2}}(x,x')dx \leq \int_{[p,q]}f_{F_{1}^{2}}(x,x')dx
\end{equation}
and 
\begin{equation} \label{eq:H2}  
\int_{\gamma}f_{\overline{F_{1}^{2}}}
(x,x')dx \leq \int_{[p,q]}f_{\overline{F_{1}^{2}}}
(x,x') dx.
\end{equation}
Adding (\ref{eq:H21}) and (\ref{eq:H2}), we get
\begin{equation}\label{eq:K1}
\int_{\gamma}f_{F_{1}^{2}}(x,x')dx+\int_{\gamma}f_{\overline{F_{1}^{2}}}
(x,x')dx\leq \int_{[p,q]}f_H(x,x')dx=H(p,q).
\end{equation}

This  shows that $H$ is timelike Finsler, with its timelike Minkowski functional at each point $x$ given by $f_H(p, v)$.

\end{proof}

The timelike Finsler structure $P_H$ is well-behaved in the sense that the linear functional 
\[
P_H(p, \cdot): C(p) \rightarrow {\mathbb R}
\]
is a timelike Minkowski functional (in the sense of Definition \ref{def:Min}) defined on the open cone $C(p) = C_2^+(p)= -C_1^+(p)$ in $T_p \Omega$.

\section{The timelike spherical relative Funk geometry}\label{s:TSRF}

 In this section and in the rest of this paper, the ambient space $\mathbb{R}^n$ is replaced by the sphere $S^n$. 
 We equip $S^n$ with its canonical metric for which it becomes a Riemannian manifold of constant curvature 1 and of diameter $\pi$. The shortest lines (geodesics) connecting two points of $S^n$ are pieces of great circles. Great circles have length $2\pi$. We first discuss a few basic notions concerning convexity on the sphere and we start with the definition:
 
 \begin{definition}[Convex subset]
 A convex subset of $S^n$ is a subset $I\subset S^n$ such that $I\not=S^n$ and such that for $x$ and $y$ in $I$, any shortest line joining them is contained in $I$.
 \end{definition}
It follows from this definition that $I$ is contained in an open hemisphere of $S^n$, that is, one of the two half-spaces bounded by a \emph{great hypersphere} $\pi$, that is, an $(n-1)$-dimensional sphere totally geodesically embedded in $S^n$. Also note that when $S^n$ is realized as the unit sphere of $\mathbb{R}^{n+1}$, then a great hypersphere of $S^n$ is the intersection of this sphere with a hyperplane passing through the origin of the coordinates.
Each great hypersphere $\pi$ has two poles.

Let $j$ be the stereographic projection  from the center of $S^n$, defined on the hemisphere $\mathbb{U}$ containing the convex set $I$, onto the tangent plane $T_N S^n \subset \mathbb{R}^{n+1}$ at the pole $N$ of $\pi$ belonging to $\mathbb{U}$.  The image $j(I)$ of the convex set $I$ is thus regarded as  a convex subset of $\mathbb{R}^n$. This projection sends the great circles of $S^n$ to the lines in $\mathbb{R}^n$, and the convexity properties of subsets of $S^n$ can be translated into convexity properties of their images by the map $j$. In particular, a subset $I$ of $S^n$ is convex if and only if its image $j(I)\subset \mathbb{R}^n$ is convex.
 
 We have to introduce some terminology regarding the sphere $S^n$ in order to define the spherical Funk geometry. This is analogous to the terminology we introduced in the Euclidean case.
 
 A \emph{supporting hyperplane} $\pi$  to an open convex subset $I$ of $S^n$ is a great hypersphere whose intersection with the closure $\overline{I}$ of $I$ is nonempty and such that $I$ is contained in one of the two connected components of the complement of $\pi$ in $S^n$. We call this component $H^+_\pi$ and we call the other component $H^-_\pi$. Each open convex subset of the sphere has a supporting hyperplane at each point of its boundary. In case we are given a collection $I_i$ of convex sets, and when we talk about their respective supporting hyperplanes and the connected components of the complements of these hyperplanes, then the upper left-side index  of $^iH_\pi^\pm$ denotes the relevant convex set $I_i$. 
 
  In the rest of this section, $I_1$ and $I_2$ are open convex subsets of $S^n$ whose bounding convex hypersurfaces are called $K_1$ and $K_2$ respectively. We have $\overline{K_i}= I_i\cup K_i$ for $i=1,2$.  We shall also say that a supporting hyperplane to $I_i$ is a supporting hyperplane to $K_i$ or to $\overline{K_i}$, depending on the subset of the sphere that we want to stress on.
  
 We shall always assume that the property expressed in the following definition is satisfied by $K_1$ and $K_2$.
 
 \begin{definition}\label{rel-pos}
 We say that the two hypersurfaces $K_1,K_2$ are \emph{in good position} if the following two properties are satisfied: 
 \begin{enumerate}
 \item $\overline{K_1}\cap \overline{K_2}=\emptyset$;
 \item For any great circle $C$ such that $C\cap K_i\not=\emptyset$ for $i=1,2$, the set $C\setminus (\overline{K_1}\cup \overline{K_2})$ is the union of two spherical segments of length $<\pi$. 
 \end{enumerate}
 \end{definition}

 \begin{proposition} \label{prop:anti} Assume $K_1, K_2$ are in good position. Then, the union  $I_1\cup I_2$ contains a pair of antipodal points, each of which belongs to one of the sets $I_1, I_2$.
 \end{proposition}
 \begin{proof}
 Take any great circle $C$ on $S^n$ intersecting the two convex sets $\overline{K_1}$ and $\overline{K_2}$. By assumption, $C$ intersects $S^n\setminus (\overline{K_1}\cup \overline{K_2})$ in two spherical segments, each of length $<\pi$. Consider one of these two segments and let $k_1\in \overline{K_2}$ and $k_2\in \overline{K_2}$ be its two boundary points. On the great circle $C$, moving monotonically $k_1$ and $k_2$ inside $I_1$ and $I_2$ respectively, we find, by continuity, two points in $I_1$ and $I_2$ on $C\cap (I_1\cup I_2)$ whose distance is equal to $\pi$.  This proves the proposition.
 \end{proof}

Let ${\mathcal P}_1$ and  ${\mathcal P}_2$ be the sets of supporting hyperplanes to $K_1$ and $K_2$ respectively, and let ${\mathbb P}_1$  and  ${\mathbb P}_2$  be respectively  the collections of great hyperspheres that do not intersect the open convex sets $I_1$ and $I_2$.  We have ${\mathbb P}_i \supset {\mathcal P}_i$. Then, we have, for $i=1,2$,
\[
I_i = \cap_{\pi \in {\mathcal P}_i} H^+_\pi = \cap_{\pi \in {\mathbb P}_i} H^+_\pi. 
\]

   We let $\Omega$ be the union of the open segments $]a_1,a_2[\in S^n$ such that $a_1\in K_1$, $a_2\in K_2$ and with $]a_1,a_2[\cap (K_1\cup K_2)=\emptyset$.
   \begin{proposition} We have 
   \[\Omega=S^n\setminus (\overline{K_1}\cup \overline{K_2}).\]
 \end{proposition}
 \begin{proof} The inclusion $\Omega\subset S^n\setminus (\overline{K_1}\cup \overline{K_2})$ is clear from the definition of $\Omega$. Let $P$ and $N$ be two antipodal points in $S^n$ contained respectively in $I_1$ and $I_2$ (Proposition \ref{prop:anti}). Given a point $p\in S^n\setminus (\overline{K_1}\cup \overline{K_2})$, consider a great circle $C$ through $N$ and $S$ containing $p$. This circle intersects $\Omega$ in two open segments, one of which contains  $p$. Let $]a_1,a_2[$ be this segment. We may assume without loss of generality that $a_i\in K_i$ for $i=1,2$.  This shows that $p$ is in $\Omega$.
 \end{proof}

 We now define a partial order relation on $\Omega$.
 
 \begin{definition}[Partial order] For $p$ and $q$ in $\Omega$, we write $p<q$, and we say that \emph{$q$ is in the future of $p$}, or that \emph{$p$ is in the past of $q$}, if there exists a segment $[p,q]$ of a great circle $C$  such that $[p,q]$ joins $p$ and $q$ and such that there exist two points $a_1\in C\cap K_1$ and $a_2\in C\cap K_2$ with the four points $a_1,p,q,a_2$ lying in that order on $C$ with $]a_1,a_2[\subset \Omega$, and $]a_1,a_2[$ not contained in any supporting hyperplane to $K_1$ or $K_2$.
 \end{definition}

 As usual, we write $p\leq q$ if $p<q$ or $p=q$.

 For any point $p$ in $\Omega$, we set  $\mathbb{P}_2(p)$ to be the set of great hyperspheres in $S^n$ separating $p$ from $I_2$.
 
\begin{definition}[Future and past]  Given a point $p$ in $\Omega$, we call the \emph{future} of $p$ the set of points $q$ in $\Omega$ such that $p<q$, and we denote this set by $\mathfrak{I}_2^+(p)$, and the
  \emph{past} of $p$ the set of points $q$ in $\Omega$ such that $q<p$, and we denote this set by $\mathfrak{I}_2^-(p)$.
  \end{definition}

  The set ${\mathfrak I}_2^+(p)$   is  an open subset of $\mathbb{R}^n$. It is also  characterized by the following: 
\begin{equation}\label{char:Ib}
{\mathfrak I}_2^+(p) = \mathrm{Int}\left(\{(\cap_{\pi \in \mathbb{P}_2^c(p)} \overline{^2H^+_\pi}) \setminus \overline{K_2} \} \cap   \{\cap_{\pi \in\mathbb{P}_1(p)} \overline{^1H^-_\pi} \} \right).
\end{equation}

  
\begin{proposition} \label{prop:eq-sphere}
We have the equivalences: 
\[
p<q \iff \mathbb{P}_2^c(p) \cup \mathbb{P}_1(p)  \subsetneq \mathbb{P}_2^c(q) \cup\mathbb{P}_1(q) \iff  \mathbb{P}_2(p) \cup \mathbb{P}_1^c(p)  \supsetneq \mathbb{P}_2(q) \cup \mathbb{P}_1^c(q).
\]

\end{proposition}

We introduce the set $\mathbb{P}_{12}(p)=  \mathbb{P}_2(p) \cup \mathbb{P}_1^c(p) $. 
The second equivalence in the above proposition becomes
\[
p<q \iff \mathbb{P}_{12}(p) \supsetneq \mathbb{P}_{12}(q).
\]

We let $\mathcal{P}_2(p)$ be the set of supporting hyperplanes to $K_2$ defined as
 \[
\mathcal{P}_2(p) =  \mathbb{P}_{12}(p) \cap \mathcal{P}_2
 \]

The set $\mathcal{P}_2(p)$ is also the set of supporting hyperplanes to $K_2$ at the points of $K_1^2(p)$, the future set of $p$ in $K_1^2$.

 We have the following:
  \begin{proposition}[Transitivity of the partial order relation]  Let $p$, $q$ and $r$ be three points in $\Omega$ satisfying $p\leq q$ and $q\leq r$. Then we have $p\leq r$.
 \end{proposition}
 \begin{proof} 
 The proof follows the same outline as Proposition \ref{p:inclusion112}.
\end{proof}

For each $p\in \Omega$, we let $\mathcal{P}_2(p)$ denote the union of the supporting hyperplanes at $K_2$ that separate $p$ from $I_2$.

The following proposition is now also proved using the methods introduced previously.
\begin{proposition}\label{prop:p2}
For any two points $p$ and $q$ in $\Omega$, we have:
\[p<q\iff \mathcal{P}_2(p)\supset \mathcal{P}_2(q).\]
\end{proposition}

  We now define the \emph{timelike spherical relative Funk distance} $F_1^2$. Its domain of definition is the subset $\Omega_{\leq}$ of the product $\Omega\times \Omega$  consisting of pairs $(p,q)$ with $p\leq q$. We are using the notation that we used in \S \ref{s:ETRF} in the context of the timelike Euclidean relative Funk geometry,  assuming that this will not cause any confusion, since the present section and \S \ref{s:ETRF} are independent.

\begin{definition}[The timelike spherical Hilbert geometry] We first define  $F_1^2$ on the subset $\Omega_<$ of $\Omega\times\Omega$ consisting of pairs $(p,q)$ with $p<q$ by the formula 
 \[F_1^2(p,q)=\log\frac{\sin d(p,b(p,q))}{\sin d(q,b(p,q))},\]
 and we then extend this definition to the pairs $(p,p)$ in the diagonal of $\Omega\times\Omega$ by setting $F_1^2(p,p)=0$ for any such pair. Here $d$ is the usual spherical distance function. 
 \end{definition}

\begin{proposition} The timelike spherical relative Funk distance is also given by:
  \[F_1^2(p,q)=\inf_{\pi\in \mathcal{P}_2(q)}\log\frac{\sin d(p,\pi)}{\sin d(q,\pi)}.\]
  \end{proposition}
  \begin{proposition}[Time inequality] The function $F_1^2(p,q)$ satisfies the timelike inequality:
  \[F_1^2(p,q)+F_1^2(q,r)\leq F_1^2(p,r) \]
  for any $p,q,r$ in $\Omega$ such that $p<q<r$. 
  \end{proposition}
  \begin{proof}  
Since  ${\mathcal P}(q) \supset {\mathcal P}(r)$ (Proposition \ref{prop:p2}), we have 
 \begin{eqnarray*}
F_1^2(p,q)+F_1^2(q,r)&=& 
\inf_{\pi \in {\mathcal{P}_2(q)}} \log \frac{\sin d(p, \pi)}{ \sin d(q, \pi)} +
\inf_{\pi \in {\mathcal{P}_2(r)}} \log \frac{ \sin d(q, \pi)}{ \sin d(r, \pi)}\\
&\leq & 
  \inf_{\pi \in {\mathcal{P}_2(r)}} \log \frac{ \sin d(p, \pi)}{ \sin d(q, \pi)} +
\inf_{\pi \in {\mathcal{P}_2(r)}} \log \frac{ \sin d(q, \pi)}{ \sin d(r, \pi)}\\
&\leq &
\inf_{\pi \in {\mathcal{P}_2(r)}}\left( \log \frac{ \sin d(p, \pi)}{ \sin d(q, \pi)}
+  \log \frac{ \sin d(q, \pi)}{ \sin d(r, \pi)}
\right) \\
&=&
  \inf_{\pi \in {\mathcal{P}_2(r)}} \log \frac{ \sin d(p, \pi)}{ \sin d(r, \pi)}\\
  &=& F_1^2(p,r).
\end{eqnarray*}
\end{proof}

The following proposition concerning the geodesics of a timelike spherical relative Funk geometry is an analogue of Proposition \ref{prop:geo2} 
 that concerns the timelike Euclidean relative Funk geometries, and it is proved in the same way. It will be useful in the next section, which concerns the Finsler structure of such a geometry. 
 
\begin{proposition}[Geodesics] \label{prop:geo27} A timelike spherical relative Funk geometry $F_1^2$ defined on a set $\Omega_\leq $ associated with two disjoint convex hypersurfaces $K_1$ and $K_2$ in $S^n$ satisfies the following:
\begin{enumerate}
\item \label{eucl} The spherical segments  in $\Omega$ that are of the form $[p,q]$ where $p<q$ are $F_1^2$-geodesics. 
\item The spherical segments in (1) are the unique $F_1^2$-geodesic segments
if and only if there is no nonempty open spherical segment contained in the subset $K_1^2$ of points in $K_2$ facing $K_1$.
\end{enumerate}
\end{proposition}

\section{The  Finsler structure of the timelike spherical relative Funk geometry}\label{s:FSTSRF}

For every point $p$ in $\Omega \subset S^n$, we associate a  \emph{timelike Minkowski functional} $f_{F_{2}^{1}}(p,v)$ defined on the subset of the tangent space $T_p\Omega$ of $\Omega$ at $p$ consisting of the zero-vector union the non-zero vectors $v$ satisfying 
\begin{equation}\label{eq:nonzero}
 \exp_p tv \in \mathfrak{I}_1^+(p) \mbox{ for some } t > 0,
\end{equation}
where $\exp_p: T_p S^n \rightarrow S^n$ is the exponential map based at $p$. We denote by  $C_1^+(p) \subset T_p \Omega$ the union of  vectors $v$ that satisfy the property (\ref{eq:nonzero}) or are the zero vector.  

From the definition of the partial order relation $p<q$ on $\Omega$, the fact that $q$ lies in the future of $p$ and $p$ lies in the past of $q$ are equivalent: 
 \begin{equation} \label{future-past-symmetry-sphere} 
q \in \mathfrak{I}_2^+(p) \Leftrightarrow p \in \mathfrak{I}_1^+(q).
\end{equation}
Thus, there is a symmetry between $C_1(p)$ and $C_2(p)$ in the sense that
 \[
 v \in C_1(p) \Leftrightarrow -v \in C_2(p).
 \]
 
\begin{definition}[Timelike Minkowski functional] We define the function $f_{F_{1}^{2}}(p,v)$ for  $p \in \Omega$ and  for any nonzero vector $v \in C_1^+(p)$  by the following formula:
\begin{equation}\label{eq:f21sphere}
f_{F_{1}^{2}}(p, v) = \inf_{\pi \in {\mathcal P}_2(p)} \frac{\langle v, \eta_\pi \rangle }{\tan d(p, \pi)} 
\end{equation}
where $\langle\cdot, \cdot \rangle$ is the canonical Riemannian metric on $S^n$, and $\eta_\pi$ the unit tangent vector at $p$ perpendicular to $\pi$ (with respect to the underlying Euclidean metric) and pointing toward $\pi$. We extend the definition by setting $f_{F_{1}^{2}}(p,0)=0$ when $v =0$.
\end{definition}

Note that due to the condition  imposed in Definition  \ref{rel-pos} on the relative position of $K_1$ and $K_2$, we have $d(p, \pi) < \pi$ for each $\pi \in {\mathcal P}_2(p)$, which in turn makes the function $f_{F_{1}^{2}}$ well-defined.

There is a timelike Minkowski functional $f_{F_{2}^{1}}$ defined on $C_1(p)$ for the timelike spherical relative Funk metric $F_2^1$, obtained simply by interchanging the indices $1$ and $2$ of $f_{F_{1}^{2}}$. 

\begin{definition}[Timelike reverse Minkowski functional]  We define the timelike Minkowski functional $f_{\overline{F_{2}^{1}}}(p,v)$ for the timelike Euclidean relative reverse Funk geometry $\overline{F_{1}^{2}}$ by
\[
f_{\overline{F_{2}^{1}}}(p,v) = f_{F_{2}^{1}}(p, -v),
\]
for $v \in C_2^+(p) = - C_1^+(p) $.
\end{definition}

Thus the two  timelike Minkowski functionals $f_{F_{2}^{1}}$ and $f_{\overline{F_{2}^{1}}}$ share the same domain of definition in $T_p S^n$.
It is easy to check the following: 

\begin{proposition}
The functionals $f_{F_{1}^{2}}(p, v)$ and $f_{\overline{F_{2}^{1}}}$ defined on the open cone $C_2^+(p)$ in $T_p \Omega$ satisfy all the properties required by a timelike Minkowski functional (Definition \ref{def:Min}).  
\end{proposition}

Repeating the arguments in \S \ref{s:Funk-Finsler}, we set up a timelike space using the Finsler structure. We say that a piecewise  $C^1$ curve  $\sigma: [0,1] \to  \Omega \subset S^n$, $t\mapsto \sigma(t) $, is {\it timelike} if at each 
time $t\in [0,1]$ the tangent vector $\sigma'(t)$ is an element of the cone $C_2^+(\sigma(t)) \subset T_{\sigma(t)} \Omega$.

\begin{definition}[The partial order relation]
Suppose that $p$ and $q$ are two points in $\Omega$. We write $p \prec q$, and we say that \emph{$q$ is in the $\prec$-future of $p$}, if there exists a timelike piecewise $C^1$ curve  $\sigma: [0,1] \to \Omega$ joining $p$ to $q$.
\end{definition}
 
By following the outline of the corresponding results proved in \S \ref{s:Funk-Finsler}, the following holds in the present setting.  (The proof is the same as that of Proposition \ref{order-equiv} except that $\mathfrak{I}^+(p)$ needs to be replaced by $\mathfrak{I}_2^+(p)$.)

\begin{proposition}
The two order relations $<$ and $\prec$ coincide; namely, for any two points $p$ and $q$ in $\Omega$, we have 
\[
p < q \Leftrightarrow p \prec q.
\]
\end{proposition} 

As we did in \S \ref{s:Timelike-Finsler}, we denote by $\delta_1^2$ the timelike intrinsic distance function associated with this timelike Finsler structure:
\begin{equation}
\delta_1^2(p, q) = \sup_{\sigma} {\rm Length}(\sigma)
\end{equation}
where the supremum is taken over all the timelike piecewise $C^1$ curves $\sigma:[0,1]\to \Omega$ satisfying $\sigma(0)=p$ and $\sigma(1)=q$.  Again, following the general set up of \S \ref{s:Timelike-Finsler}, we show, as in Lemma \ref{lemma:finite}, that the intrinsic distance $\delta(p, q)$ for $p<q$ is finite. Finally, we obtain that the domain of definition $\Omega_<$ defined with the partial order $<$ for  the timelike Funk distance $F_1^2$  and the domain of definition $\Omega_\prec$ for the  timelike distance function $\delta_1^2$  coincide. Furthermore, we shall prove the equality $\delta_1^2(p, q) =F_1^2 (p, q)$ for any pair $p < q$ in $\Omega_<$. We state this as follows:

\begin{theorem}\label{th:distance1}
For a pair $(p,q)\in \Omega_{\leq}$, we have
\[
F_1^2(p, q) = \delta_1^2(p, q).
\]
\end{theorem}

In different words, we have the following useful form of Theorem \ref{th:distance1}:
\begin{theorem}\label{Funk-Finsler-3} 
The  timelike spherical relative Funk geometry $F_1^2$ is a timelike Finsler structure with associated 
Minkowski functional $f_{F_{1}^{2}}(p,v)$.
\end{theorem}

With the identification $F_1^2=\delta_1^2$, given a pair of points $p, q$ with $p < q$, there always exists a $\delta$-distance-realizing (length-maximizing) geodesic from $p$ to $q$, since the spherical geodesic   $[p,q]$ is an $F_1^2$-geodesic.

\begin{proof}[Proof of Theorem \ref{th:distance1}] The proof is similar to the proof of Theorem \ref{th:distance}: 
Given a pair $(p, q)$ with $p<q$, consider  the spherical geodesic ray $R(p, q)$ from $p$ through $q$ and let $b(p, q) \in K_2$ be the first intersection point of this ray with the convex set  $\overline{K_2}$. Parameterize the geodesic segment $[p,q]$  by a path $\sigma(t)$ having parameter $t$ proportionally to arc-length, with $\sigma(0)=p, \sigma(1)=q.$  Then we have 
\[
\int_0^1 f_{F_1^2}(\sigma(t), \sigma'(t)) \, dt = \log \frac{\sin d(p, b(p,q))}{\sin d(q, b(p,q))} = F_1^2(p, q),
\] 
since 
\[
\frac{d}{dt} \log \frac{\sin d(p, b(p,q))}{\sin d(\sigma(t), b(p,q))} = f_{F_1^2}(\sigma(t), \sigma'(t)).
\]
Taking the supremum over the set of piecewise-$C^1$ timelike paths from $p$ to $q$, we obtain the inequality
\[
\delta_1^2(p, q) \geq  F_1^2(p, q).
\]

We now need a monotonicity lemma for the intrinsic distances.  

We consider a pair of  open convex set $I_1$ and $I_2$ bounded respectively by the two disjoint convex hypersurfaces $K_1$ and $K_2$  which we assume as before to be in good position (Definition \ref{rel-pos}). Let $f_{F_{1}^{2}}: T\Omega \rightarrow {\mathbb R}$ be the associated timelike Minkowski functional, and $\delta_1^2$ the intrinsic distance  induced by $f_{F_{1}^{2}}$.

Finally, let 
$\widehat{I}_2 \supset I_2$ be another open convex set and $\widehat{K}_2$ its bounding hypersurface. We assume that $K_1$ and $ \widehat{K}_2$ are also in good position. Let $f_{\widehat{F}_{1}^{2}}$  be the timelike Minkowski functional of the timelike relative Funk distance $\widehat{F}_{1}^{2}$ associated with the pair $(I_1,\widehat{I}_2)$ and $\widehat{\delta}_1^2$ the associated intrinsic distance. (Note here that the domains of definition of   $f_{\widehat{F}_{1}^{2}}$ and $\widehat{\delta}_1^2$ contain those of $f_{F_{1}^{2}}$ and $\delta_1^2$ respectively.)  

\begin{lemma}\label{lem:another3} Suppose that $(p,q)$ is in the domain of definition of both timelike intrinsic distances $\delta_1^2$ and $\widehat{\delta}_1^2$. Then we have
\[
\widehat{\delta}_1^2(p, q) \geq \delta_1^2(p, q).
\]
\end{lemma}

The proof is, with an adaptation of the notation, the same as that of Lemma \ref{lem:another1}.

\noindent {\it Proof of Theorem  \ref{th:distance1} continued}.--- For $(p, q) \in \Omega_<$,  let 
\[\widehat{I}_2 = H^+_{\pi_{b(p, q)}},\]
 where $H^+_{\pi_{b(p, q)}}$  is 
the open hemisphere bounded by a hyperplane $\pi_{b(p, q)}$ supporting $K_2$ at $b(p, q)$ and containing $\widetilde{I}_2$. The open set $\widehat{I}$ is equipped with its intrinsic distance $\widehat{\delta}$. We now apply Lemma \ref{lem:another3} to this setting, where a convex set  $\widehat{I}_2$  
contains $I_2$, and obtain $\widehat{\delta}_1^2 \geq \delta_1^2$.

For the open hemisphere $\widehat{I}_2 = H^+_{\pi_{b(p, q)}}$, the values of $F_1^2(p, q)$, $\widehat{F}_1^2(p, q)$ and $\widehat{\delta}_1^2(p, q)$ all coincide. Indeed the set $
\widehat{\mathcal{P}}_2$ of supporting hyperplanes consists of the single element $\pi_{b(p, q)}$, and the path
$\sigma$ from $p$ to $q$ is length-maximizing, since every timelike path for $\widehat{I}_2$ is $\widehat{F}_1^2$-geodesic. This follows from the considerations in \S \ref{s:ETRF}.
 
By combining the above observations, we have 
\[
 F_1^2(p, q) = \widehat{F}_1^2(p, q) = \widehat{\delta}_1^2(p, q) \geq \delta_1^2(p, q) \geq F_1^2(p, q)
\]
and  the equality $\delta_1^2(p, q) =F_1^2(p, q)$ follows.
\end{proof}

\section{The  timelike spherical Hilbert geometry}
 
 In this section, we continue using the notions and notation of \S \ref{s:TSRF}: $I_1,I_2$ is an ordered pair of convex subsets of the sphere $S^n$ whose boundaries are convex hypersurfaces in $S^n$ denoted by $K_1$ and $K_2$ respectively, satisfying the conditions stated at the beginning of that section and with $\overline{K_i}= I_i\cup K_i$ for $i=1,2$. The subset $\Omega$ of $S^n$ is defined as in \S \ref{s:TSRF}, and the partial order relation $p<q$ for $p$ and $q$  in $\Omega$ is defined accordingly, $K_1$ representing the past and $K_2$ the future. 
  
  We denote, as usual, the set of points $(p, q)$ in $\Omega \times \Omega$ satisfying $p < q$ by $\Omega_<$. 
 We also write $p \leq q$ when $p<q$ or $p=q$.  
 
 $F_1^2$ is the timelike spherical Funk metric associated with the ordered pair $K_1,K_2$. We showed that this is a timelike Finsler metric, and its associated timelike Minkowski functional, denoted by $f_{F_{1}^{2}}$ is defined for each point $p$ in $\Omega$ as in \S \ref{s:FSTSRF} on a subset of the tangent space $T_p\Omega$ of $\Omega$ at $p$ which is a cone denoted by $C_2^+(p)$.

 As in the Euclidean case (see Definition \ref{def:123}), there is a timelike  spherical relative reverse Funk metric $\overline{F_1^2}$ associated with the pair $(K_1,K_2)$. For this, we first consider the timelike spherical relative Funk metric $F_2^1$ associated with the ordered pair $(K_2,K_1)$, and we define the new function $\overline{F_1^2}$, whose domain of definition  is equal to the domain of definition of $F_1^2$, by
 \[\overline{F_1^2}(p,q)=F_2^1(q,p) .\]

 \begin{definition}[Timelike spherical Hilbert metric]
 The timelike spherical Hilbert metric $H_1^2$ associated with the ordered pair $K_1,K_2$ is defined on the set of ordered pairs $(p,q)$ such that $p<q$ in the setting where the convex set $K_1$ represents the past and the convex set $K_2$ the future, by the formula 
 \[ H(p,q)=\frac{1}{2}(F_1^2(p,q)+\overline{F_1^2}(p,q)).\]

As usual, the definition is extended to the case where $p=q$ by setting $H(p,q)=0$.
 \end{definition}

Unlike the situation studied in \cite{PY2}, there is no straightforward way of defining a timelike Funk spherical metric,. One reason is that given two distinct points in the complement of a convex subset of the sphere $S^n$, there is no natural way of saying that one is in the future of the other (the great circle through these  points may intersect the convex set in two points).

 We recall that given four points $p_1, p_2, p_3, p_4$ situated in that order on a great circle on the sphere, their spherical cross ratio is defined by
\[
[p_1, p_2, p_3, p_4] = \frac{\sin d(p_2, p_4) \sin d(p_3, p_1)}{\sin d(p_3, p_4) \sin d(p_2, p_1)}.
\] 
Its values are in ${\mathbb R}_{\geq 0} \cup \{\infty\}$. The spherical cross ratio is a projectivity  invariant, cf. \cite{PY2}. 

For any pair of points $(p, q)$ in $\Omega_<$, let $a_1 \in K_1$ and $a_2 \in K_2$ be the intersection points between the great circle 
through $p$ and $q$ and the two hypersurfaces $K_1$ and $K_2$, so that $a_1, p, q, a_2$ lie on the arc of great circle $[a_1, a_2] \subset \Omega$ in that order.   With this notation, the timelike spherical Hilbert distance associated with the pair $(K_1,K_2)$  is also given by the following equivalent form:
\begin{proposition}\label{prop:SH2} Let $p$ and $q$ be two points in $\Omega$ satisfying $p<q$ and let $[a_1,a_2]$ be the segment of great circle containing $p$ and $q$ with $[a_1,a_2]\cap K_i=a_i$  for $i=1,2$. Then, we have:
\[
H(p, q) =\frac{1}{2} \log [a_1, p, q, a_2].
\]
\end{proposition}

\begin{proposition}[Invariance] \label{prop:SHinv}  
The timelike spherical Hilbert geometry associated with the pair of convex sets $K_1,K_2\subset S^n$ is invariant by the projective transformations of the sphere $S^n$ that preserve setwise each of the two convex sets $K_1,K_2$. \end{proposition}

The timelike spherical Hilbert geometry $H$ has an underlying timelike Finsler structure which we describe in the next section. For that, we need first to talk about $H$-geodesics. 
The following proposition is analogous to Proposition  \ref{Hibert-geo} concerning the  timelike Hilbert geometry.

\begin{proposition}\label{Spherical-Hibert-geo}
(a) In a timelike spherical Hilbert geometry $H$ associated with an ordered pair of convex hypersurfaces $K_1,K_2$, the spherical segments of the form $]a_1,a_2[$, equipped with their natural orientation from $a_1$ to $a_2$ and satisfying the following three properties
\begin{enumerate}
\item \label{pr1} $a_1 \in K_1$ and $a_2 \in K_2$; 
\item $]a_1,a_2[$  is not contained in any supporting hyperplane to $K_1$ or to $K_2$;
\item  \label{pr3} the open spherical segment $]a_1,a_2[$ is in the complement of $K_1\cup K_2$ 
\end{enumerate}
 are $H$-geodesics. Each such geodesic (with its orientation) is isometric to the real line.
 
 (b) The oriented spherical segments contained in the segments  of the form $[a_1,a_2]$  satisfying the properties (\ref{pr1}-\ref{pr3}) above are the unique $H$-geodesics
if and only if the following holds: There are no spherical segments $[a_1,a_2]$ of the above form with $a_1$ in the interior of an open nonempty spherical segment $J_1\subset K_1$ and $a_2$ in the interior of an open nonempty segment $J_2\subset K_2$, with $J_1$ and $J_2$ coplanar (contained in a 2-dimensional sphere).
\end{proposition}
The proof is an adaptation of that of the non-timelike spherical Hilbert metric (Proposition 8.2 of \cite{2012-Hilbert2}), and we omit it.
 
We end this section by a remark concerning the hyperbolic analogues of our timelike spherical Hilbert geometry.

\begin{remark}[Timelike hyperbolic Funk geometry and timelike hyperbolic Hilbert geometry]
Let us first recall that there is a Funk geometry associated with a convex hypersurface $K$ in the hyperbolic space $\mathbb{H}^n$. This was studied by the authors in \cite{PY2}. In the same way, one can define a timelike Funk geometry associated with convex subsets of $\mathbb{H}^n$. 
The pre-order $p < q$ is defined as in the case of the timelike Euclidean Funk geometry, and the timelike distance from $p$ to $q$, where $p$ and $q$ satisfy $p<q$, is given by the formula
\begin{equation}\label{eq:hyp.FF}
F(p, q) = \log \frac{\sinh d(p, b(p, q))}{\sinh d(q, b(p, q))} 
\end{equation}
where $b(p, q)$ is the point where the ray $R(p, q)$ hits $K$ for the first time, and $d$ is hyperbolic distance.  
 Several properties of the hyperbolic (non-timelike) Funk metric proved in \cite{PY2} hold \emph{verbatim} for this timelike hyperbolic Funk geometry. 
In particular, we have a variational formulation of the timelike hyperbolic Funk distance:
\[
F(p, q) = \inf_{\pi \in {\mathcal P}(p)} \log \frac{\sinh d(p, \pi)}{\sinh d(q, \pi)}.\]

There is also a timelike hyperbolic Hilbert geometry, defined in a way analogous   to the timelike Hilbert geometry defined in \S \ref{def:H}, replacing, in the definition, the distance by the hyperbolic sine of the distance, as we did in the definition of the timelike hyperbolic metric in (\ref{eq:hyp.FF}).

The hyperbolic segments are geodesics for the timelike hyperbolic Funk and for the timelike hyperbolic Hilbert geometries.
\end{remark}

\section{The timelike Finsler structure of the timelike spherical Hilbert geometry} 
We shall define a function $f_H(p,v)$ which will play the role of a timelike Minkowski functional associated with the timelike spherical Hilbert geometry $H$. It is defined on pairs $(p,v)$ belonging to the tangent bundle of $\Omega$, where $p\in \Omega$ and $v$ is a vector in the tangent space $T_p\Omega$ which is either the zero vector or a vector tangent to a segment of great circle starting at $p$ and pointing in the direction of a point in $\mathfrak{I}^+(p)$. This function $f_H(p,v)$ is defined by the same formula as the timelike Minkowski norm associated with the Finsler structure of the Euclidean Hilbert geometry (see formulae (\ref{eq:Hilbert2}) and (\ref{eq:Hilbert1})):
\begin{equation} \label{eq:Hilbert21}
f_H(p, v) = f_{F_{1}^{2}}(p,v) + f_{F_{2}^{1}}(p,-v),\end{equation}
or, equivalently,
\begin{equation} \label{eq:Hilbert11}
f_H(p, v) = f_{F_{1}^{2}}(p,v) + f_{\overline{F_{1}^{2}}}(p,v)\end{equation}
where $f_{F_{1}^{2}}$ and $f_{\overline{F_{1}^{2}}}$ are now the timelike Minkowski norms on the tangent spaces associated with the timelike spherical relative Funk geometry and the  timelike reverse Funk geometry associated, as in \S \ref{s:FSTSRF}. The convex hypersurfaces $K_1$ and $K_2$ with the given order are implied by the notation. 
 
Repeating the argument in \S \ref{s:Timelike-Finsler}, we set up a timelike distance function using the Finsler structure $f_H$. We say that a piecewise  $C^1$ curve  $\sigma: [0,1]\to  \mathbb{R}^n$, $t\mapsto \sigma(t)$ is {\it timelike} if at each 
time $t\in J$ the tangent vector $\sigma'(t)$ is an element of the cone $C_2^+(\sigma(t)) \subset T_{\sigma(t)} \mathbb{R}^n$.

\begin{definition}[The partial order relation]
Suppose that $p$ and $q$ are two points in $\mathbb{R}^n$. We write $p \prec q$, and we say that \emph{$q$ is in the $\prec$-future of $p$}, if there exists a timelike piecewise $C^1$ curve  $\sigma: J \to \Omega$ joining $p$ to $q$.
\end{definition}
 
By following the outline in \S \ref{s:Timelike-Finsler}, we have the following in the present setting.  The proof is the same as that of Proposition \ref{order-equiv} except that $\mathfrak{I}^+(p)$ needs to be replaced by $\mathfrak{I}_2^+(p)$. 

\begin{proposition}
The two order relations $<$ and $\prec$ coincide; namely, for any two points $p$ and $q$ in $\Omega$, we have 
\[
p < q \Leftrightarrow p \prec q.
\]
\end{proposition} 

In analogy with the previous settings considered, we now denote by $\delta$  the timelike intrinsic distance function associated with this timelike Finsler structure:
\begin{equation}
\delta(p, q) = \sup_{\sigma} {\rm Length}(\sigma)
\end{equation}
where the supremum is taken over all the timelike piecewise $C^1$ curves $\sigma:[0,1]\to \Omega$ satisfying $\sigma(0)=p$ and $\sigma(1)=q$.  Also, by the work done in \S \ref{s:Timelike-Finsler} (see Lemme \ref{lemma:finite}), the intrinsic distance $\delta(p, q)$ for $p<q$ is finite. The domain of definition $\Omega_<$ defined with the partial order $<$ for  the timelike Hilbert distance $H$  and the domain of definition $\Omega_\prec$ for the  timelike distance function $\delta$  coincide. 

 The following theorem is then proved in the same way as Theorem \ref{Hilbert-Finsler}, replacing, in the proof, the Euclidean segments joining a pair $p,q$ by the spherical geodesic joining them:

\begin{theorem}\label{Hilbert-FinslerA} 
The timelike spherical Hilbert geometry has an underlying timelike Finsler structure given by the 
Minkowski functional $f_H$ defined in (\ref{eq:Hilbert11}).
\end{theorem}

\section{Timelike spherical Hilbert geometry with antipodal symmetry: Light cone and null vectors}\label{s:anti}

We consider a notable case of a timelike spherical Hilbert metric, namely, the case where the underlying two convex hypersurfaces $K_1$ and $K_2$ are antipodal in $S^n$, that is, they satisfy $K_2 = - K_1$ where the 
minus sign refers to the antipodal map $x \mapsto -x$ of $S^n$ 
modeled in ${\mathbb R}^{n+1}$.  Note that the antipodality condition guarantees that $K_1$ and $K_2$ are in good position (Definition \ref{rel-pos}) on $S^n$.  

In this geometry, the quotient space by the antipodal symmetry group $\mathbb{Z}_2$ is identified with a timelike Hilbert geometry on an open subset of   the projective space $
{\mathbb R}{\rm P}^n$, in which $K_1$ and $K_2$ become 
a single convex hypersurface $K$ under the quotient map $S^n \rightarrow {\mathbb R}{\rm P}^n$.  This has been investigated by Busemann \cite{B-Timelike}. We do not, however,  consider the projective space here, and exclusively  treat  the spherical setting with two convex sets $\overline{K_1}$ and $\overline{K_2}$.  (Working in the projective space, Busemann gets locally timelike spaces instead of timelike spaces.)
%

In this setting, there is a doubling phenomenon for the rays emitted from a point $p \in S^n$ in 
the complement $\Omega$ of $\overline{K_1} \cup \overline{K_2}$ where $\overline{K_i} = K_i \cup I_i$: if such a ray intersects $K_2$ at a point $K_2^+(p)$ in the future, then it also  does so at a point $K_1^-(p)$ in the past.

   Let us recall that in the physics modeled by Minkowski geometry, the fact that a curve in the light cone has zero length corresponds to the fact that light travels along it at infinite speed. So far, we have carefully avoided the issue of null vectors in  timelike geometry.  (This is our condition that in the definition of relation $p<q$ we do not allow the pair $p,q$ to be on a supporting line of the convex set.) We did so because there is no obvious coherent general treatment of such vectors in the timelike Funk and Hilbert geometries. However, this setting,  where $K_1$ and $K_2$ are antipodally located on $S^n$,  is a particular situation worth being investigated in which null vectors arise.  The details are as follows.
   
   A great circle intersecting $K_1$ at two points $a_1, b_1$ also intersects $K_2$ at two antipodal points $a_2(=-b_1)_1, b_2(=-a_1)$.    Now consider the situation where a great circle $C$ is a supporting line of $K_2$ and let $a_2$ be a point in $\pi \cap K_2$.  Then, this circle $C$ is also a supporting line at $a_1$, which is identified with $-a_2$.  We consider a pair of points $p, q$ on an arc of the great circle $C$ in $\Omega$, and the timelike Hilbert distance $H(p, q)$, which is the logarithm of the spherical cross ratio of the  quadruple $(a_1, p, q, a_2)$ lying on the arc in that order.  
\[
H(x, y) = \frac12 \log \frac{\sin d(p, a_2) \sin d(q, a_1) }{\sin d(q, a_2) \sin d(p, a_1)}$$
$$=  \frac12 \log \frac{\sin d(p, a_2) \sin (\pi - d(q, a_2) ) }{\sin d(q, a_2) \sin (\pi- d(p, a_2))} = 0.
\]
Here we have used the fact that $d(p, a_1) = \pi - d(p, a_2)$, as $a_1$ is antipodal to $a_2$.
As the choices of $p$ and $q$ on the great circle $C$ are arbitrary, we conclude that the (naturally extended) timelike Minkowski functional evaluated along the supporting great circle to $K_1$ and $K_2 = - K_1$ is zero. 

 In other words, given a point $p$ in $\Omega$, consider the cone $\rm{Cone}_2(p)$  consisting of great circles through $p$ each of which is   a supporting line of $K_2$.  These great circles are automatically elements of $\rm{Cone}_1(p)$. Recall that the set of vectors  in $T_p \Omega$ on which the Minkowski functional $P_H(p)$ is defined is equal to $C_2(p)$.  Then the tangent vectors  in $T_p \Omega$ which lie in the boundary of the open cone $C_2(p)$ constitute the future-directed light cone at $p$ with respect to the timelike Minkowski functional for the  timelike spherical Hilbert geometry $H$.  
In this way, we see that null vectors in the timelike spherical Hilbert geometry with antipodal symmetry naturally exist.

  \section{The de Sitter geometry as a timelike spherical Hilbert geometry with antipodal symmety}

In this last section we explain that the de Sitter space is a special case of the timelike spherical Hilbert geometry with antipodal symmetry. 
In the setting described in the preceding section, if we take $K_1$ to be a small circle of radius $\pi/4$ in $S^{n} 
\subset {\mathbb R}^{n+1}$, then the resulting timelike Hilbert geometry is isometric to the de Sitter metric 
restricted to the timelike vectors.   We now establish this isometry.

We first recall that the $n$-dimensional de Sitter space is the unit sphere in the Minkowski space-time ${\mathbb R}^{n, 1}$, namely,
\[
{\rm dS}^{n-1, 1} = \{ (x_0, x_1, \dots, x_n) \,\, | \,\, -x_0^2 + \sum_{i=1}^n x_i^2= 1  \} \subset {\mathbb R}^{n, 1},
\]
equipped with the so-called de Sitter metric,  a Lorentzian metric of type $(n, 1)$ whose first fundamental form is induced from the ambient Minkowski metric $ds^2 = -dx_0^2 + \sum_i^n dx_i^2$. It is diffeomorphic to $S^{n-1} \times {\mathbb R}$. The de Sitter space ${\rm dS}^{n-1, 1} $ is then an $n$-dimensional Lorentzian manifold, with global time orientation where we take the future direction to be the globally defined non-vanishing vector field $\frac{\partial}{\partial x_0}$. Naturally this induces an order relation in the sense that $q$ lies in the future of $p$ when there exists a piecewise $C^1$ timelike curve from $p$ to $q$.

\begin{figure}[!ht] 
\centering
 \psfrag{1}{\small $1$}
  \psfrag{2}{\small $(-1,1)$}
   \psfrag{3}{\small $(1,1)$}
    \psfrag{4}{\small  $(1,\widehat{s_2})= P_{dS}(q)$}
     \psfrag{5}{\small  $(1,\widehat{s_1})= P_{dS}(p)$}
      \psfrag{6}{\small  $q=(\cosh t_2,\sinh t_2)$}
       \psfrag{7}{\small  $p=(\cosh t_1,\sinh t_1)$}
        \psfrag{8}{\small  $p_S^{-1}\circ p_{dS}(q)$}
         \psfrag{9}{\small  $p_S^{-1}\circ p_{dS}(p)$}
\includegraphics[width=1\linewidth]{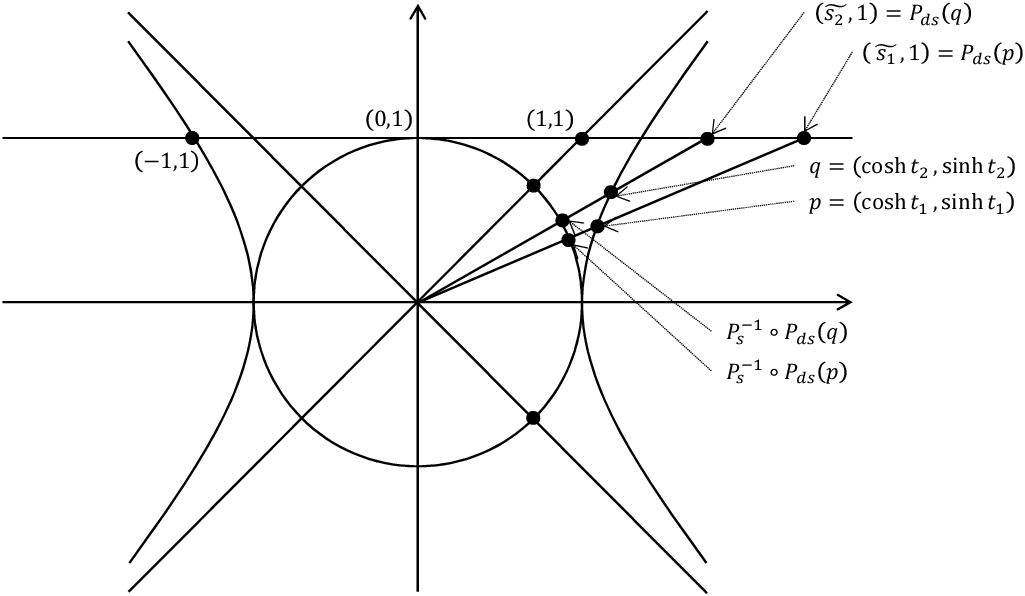}    \caption{\small {Projection from the origin of the coordinates onto the hyperplane at hight 1}}   \label{Sitter}  
\end{figure}

The intersection between the unit sphere ${\rm dS}^{n-1, 1}$ and the $x_0x_1$-plane in ${\mathbb R}$ is denoted by ${\rm dS}^{0, 1} \subset {\mathbb R}^{1,1}$. This is a totally geodesically  embedded submanifold and geometrically it is a hyperbola (see Figure \ref{Sitter}) diffeomorphic to $S^0 \times \mathbb{R}$.    By using an element 
of the orthogonal group $\mathrm{SO}(n, 1)$,  any pair $(p, q)$, with $q$ lying in the future of $p$  in 
$S^{n, 1}$  can be isometrically transposed to a pair of points on $\rm{dS}^{0,1}$ so that the $x_0$ coordinates of the points are 
positive.  Hence we may assume without loss of generality that  $p$ and  $q$ belong to a connected component of the upper hemisphere $ \mathbf {U}= \{(x_0, x_1)  | -x_0^2 + x_1^2 = 1, x_0 >0 \}$
of $\rm{dS}^{0,1}$ in ${\mathbb R}^{1,1}$.  

We introduce a parameterization $\sigma(t)$  of $\rm{dS}^{0,1}$, $t\in\mathbb{R}$, so that 
\[
(x_0, x_1) = (\sinh t, \cosh t).
\]
Note that $t$ is an arc-length parameter for the de Sitter metric, as the tangent vector to $\sigma(t) = (\sinh t, \cosh t)$ has norm 1.  Hence, for $p = \sigma(t_1)$ and $q= \sigma(t_2)$ with $t_1 < t_2$, the de Sitter distance $d(p, q)$ is equal to $t_2-t_1$.  Here the point $q$ lies in the future of $p$ in ${\rm dS}^{n-1, 1} $.

We now project, as pictured in Figure \ref{Sitter},  a part of the hyperboloid $\{(x_0, x_1) | -x_0^2+x_1^2=1, x_0 > 0 \}$ onto the hyperplane $\{x_0=1\}
$ along the rays from the origin of ${\mathbb R}^{1,1}$
\begin{equation}\label{e:P}
P_{\rm dS}:  \{-x_0^2+x_1^2=1\} \rightarrow \{x_0=1\}.
\end{equation}
Let $\widetilde{p} = (1, \widetilde{s_1})$ and $\widetilde{q} =(1, \widetilde{s_2})$ 
be the images of $p$ and $q$ by this correspondence, where  $\widetilde{s_1}>\widetilde{ s_2}$.  The asymptotic lines $x_0 = \pm x_1$ of the 
hyperboloid $\{-x_0^2+x_1^2=1\}$  are sent to the points $(1, 1)$ and $(1, -1)$.  The cross ratio of those four points defines the Hilbert geometry $H$ for the convex set $I = \{ x_0 > \pm x_1\}$  in the projective space 
${\mathbb R} P^1$, and for a pair of points $\widetilde{p}$ and $\widetilde{q}$ with $\widetilde{p} <  \widetilde{q}$, we have 
\[
H(\widetilde{p}, \widetilde{q}) = \frac12  \log \frac{\widetilde{s_1} -1}{\widetilde{s_2}-1} \cdot \frac{\widetilde{s_2}+1}{\widetilde{s_1}+1}.
\]
By noting the equality
\[
\widetilde{s_i} = \frac{\sinh t_i}{\cosh t_i},
\] 
the Hilbert distance $H(\widetilde{p}, \widetilde{q})$ is
equal to $(t_2-t_1)$.  Hence we have shown that $d(p,q) =  H(\widetilde{p}, \widetilde{q})$ for $p < q$.

  By post-composing the map $P_{\rm dS}$ with the map $P_{\rm S}
^{-1}: \{x_0 =1\} \rightarrow {\mathbb U}$ where 
${\mathbb U}$ is the upper hemisphere $\{ (x_0, x_1, \dots, x_n) \,\,|\,\, x_0^2 + \sum_{i=1}^n x_i^2 = 1, x_0 > 0\}$, the geodesic through $p$ and $q$ in the de Sitter space is identified with a great circle in the 
sphere, and the image of the map $P_{\rm S}^{-1} \circ P_{\rm dS}$ of the northern half of the de Sitter 
space is ${\mathbb U} \setminus B$ where $B$ is the northern cap bounded by the small circle of radius $
\pi/4$ (see Figure \ref{Sitter}).  This shows that the timelike geometry of the de Sitter space is realized by the timelike Hilbert 
metric modeled on the sphere.  The maps $P_{\rm dS}$ and $P_{\rm S}$ are {\it perspectivities}, namely, they preserve the cross ratio (see \cite{PY2}). We conclude that the de Sitter distance is equal to the timelike spherical Hilbert distance.

The quotient space of the de Sitter space is equipped with a {\it locally timelike} Hilbert geometry, where the 
quotient is taken by the ${\mathbb Z}_2$ antipodal symmetry of $\Omega = S^n \setminus (\overline{K_1} 
\cup -\overline{K_1})$, with $K_1$ a small circle of radius $\pi/4$ in $S^{n} $.  The timelike Hilbert geometry thus defined is only {\it local}, as the
space $\Omega = {\mathbb R}P^n \setminus \overline{K_1}$ is not time-orientable.  Namely consider the closed 
path from $p \in \Omega$ to itself, along the circle at infinity of ${\mathbb R}P^n$. Traversing the loop then reverses the 
orientation of the light cone  (cf. Hawking-Ellis \cite{HE}, Calabi-Marcus \cite{CM}).

 \medskip
 
 \emph{Acknowledgements} The first author is supported by the French ANR Grant FINSLER and by the  U.S. National Science Foundation grants DMS 1107452, 1107263, 1107367 ``RNMS: Geometric structures And Representation varieties" (the GEAR Network) and JSPS Invitation Fellowship (short-term) FY2017. He is also grateful to Gakushuin University (Tokyo), to the Graduate Center of the City University of New York and to the Tata Institute of Fundamental Research (Bombay) where part of this work was done. The second author is supported by JSPS KAKENHI 24340009, 16K13758 and 17H01091.  He is also grateful for IRMA (Strasbourg), where part of the work was done.
The authors would like to thank V. N. Berestovskii for his valuable comments on an earlier version.


\begin{thebibliography}{99}
 
 \bibitem{Alex1} A. D. Alexandrov, A contribution to chronogeometry. Canad. J. Math. 19 (1967), 1119--1128.


\bibitem{Alex2}  A. D. Alexandrov,  Mappings of spaces with families of cones and spacetime transformations.
Ann. Mat. Pura Appl. (4) 103 (1975), 229--257.

\bibitem{BeemPhd} J. K. Beem, Synthetic theory of indefinite metric spaces, PhD thesis,  University of Southern California, 1968. 

\bibitem{Beem} J. K. Beem, M. A. M. A. Kishta, and S. S. Chern,  Indefinite Finsler spaces and timelike spaces,  
Indiana Univ. Math. Jour. Vol. 23, No. 9 (1974), pp. 845--853. 

\bibitem{G}  H. Busemann, {\it The Geometry of Geodesics}, Academic Press, Now York, 1955. Reprinted by Dover, 2005 and later editions.

\bibitem{B-Axioms}  H. Busemann and J. K. Beem,   Axioms for indefinite metrics. 
Rend. Circ. Mat. Palermo (2) 15 (1966) 223-246. 

\bibitem{B-Timelike}  H. Busemann,  Timelike spaces. 
Dissertationes Math. Rozprawy Mat. 53 (1967) 52 pp. 

\bibitem{CM}  E. Calabi and L. Marcus, Relativistic space forms,  Ann. of Math. 75, (2012) 63--76. 


\bibitem{DGK}  J. Danciger, F. Gu\'{e}ritaud and F. Kassel, Geometry and topology of complete Lorentz spacetimes of constant curvature. Ann. Sci. \'Ec. Norm. Sup\'{e}r. (4) 49 (2016), no. 1, 1--56.



\bibitem{FS} F. Fillastre and A. Seppi, Spherical, hyperbolic and other projective geometries: convexity, duality, transitions, In: Eighteen essays on non-Euclidean geometry, V. Alberge and A. Papadopoulos (eds.), European Mathematical Society Publishing House, 2019, p. 321--409.

\bibitem{Gold} W. M. Goldman,
Crooked surfaces and anti-de Sitter geometry. Geom. Dedicata 175 (2015), 159--187.

\bibitem{HE} S. W. Hawking and G. F. R. Ellis, {\it The large Scale structure of space-time}, Cambridge University Press, 1973.

\bibitem{Minguzzi} E. Minguzzi, Light cones in Finsler spacetime, Comm. Math. Phys. 334 (2015) no.3, 1529--1551.

\bibitem{Woo-PdD} P. Y. Woo, Doubly-timelike general G-spaces, PhD thesis,  University of Southern California, 1968. 

\bibitem{Pap-H} A. Papadopoulos, Hilbert's fourth problem
  In \emph{Handbook of Hilbert Geometry} (A. Papadopoulos and M. Troyanov, ed.) European Mathematical Society Publishing House, Z\"urich, 2014, 
 391-431.
 
\bibitem{PT} A. Papadopoulos and M. Troyanov, Weak Finsler structures and the Funk weak metric. Math. Proc. Cambridge Philos. Soc. 147 (2009), no. 2, 419--437.

\bibitem{PT1} A. Papadopoulos and M. Troyanov, From Funk to Hilbert geometry. Handbook of Hilbert geometry (A. Papadopoulos and M. Troyanov, ed.), European Mathematical Society Publishing House, IRMA Lectures in Mathematics and Theoretical Physics, Vol. 22, p. 33-68, 2014. 
 
 \bibitem{2012-Hilbert2} A. Papadopoulos and  S. Yamada, Funk and Hilbert geometries in spaces of constant curvature. Handbook of Hilbert geometry, European Mathematical Society Publishing House, p. 330-354, 2014.

\bibitem{PY2} A. Papadopoulos and S. Yamada, On the projective geometry of constant curvature spaces,
Sophus Lie and Felix Klein: The Erlangen Program and Its Impact in Mathematics and Physics, (2015) 237--246, European Mathematical Society.



\bibitem{PY1} A. Papadopoulos and S. Yamada, Busemann's metric theory of timelike spaces. In:  The Selected works of Herbert Busemann (ed. A. Papadopoulos), Vol. I, Springer Verlag, p. 115--131,  2018.

 
\bibitem{Penrose} R. Penrose, {\it Techniques in differential topology in relativity,}  CBMS-NSF Regional Conference Series in Applied Mathematics, (1974).

\bibitem{Poin} H. Poincar\'{e}, Sur les hypoth\`eses fondamentales de la g\'eom\'etrie. \emph{Bull. Soc. Math. France} 15 (1887), 203--216.

 \bibitem{Sch} J.-M. Schlenker, Vari\'et\'es Lorentziennes plates vues comme limites de vari\'et\'es anti-de Sitter [d'apr\`es Gu\'eritaud et Kassel]. Ast\'erisque No. 380; S\'eminaire Bourbaki. Vol. 2014/2015 (2016), Exp. No. 1103, 437--497.
 
\bibitem{Study} E. Study, Beitr\"age zur nichteuklidische Geometrie.  I.  
American Journal of Math., Vol. 29, No. 2 (April 1907), pp. 101--116, English translation 
by Annette A'Campo-Neuen:  Contributions to non-Euclidean Geometry I. In: Eighteen essays on non-Euclidean geometry, V. Alberge and A. Papadopoulos (eds.), European Mathematical Society Publishing House,  Z\"urich, 2019, p. 237--251.



\bibitem{Y1} S. Yamada, Convex bodies in Euclidean and Weil-Petersson geometries, Proc. Amer. Math. Soc. 142 (2014) 603--616.

 \end{thebibliography}
\end{document}